\documentclass[11pt,a4paper]{article}

\usepackage[T1]{fontenc}
\usepackage[utf8]{inputenc}
\usepackage{amsmath,amssymb,amsthm}
\usepackage{graphicx}
\usepackage{booktabs}
\usepackage{array}
\usepackage{siunitx}
\usepackage[margin=2.6cm]{geometry}
\usepackage{xcolor}
\usepackage{stmaryrd}
\usepackage{tikz}
\usepackage{mathrsfs}
\usetikzlibrary{arrows.meta,calc}
\usepackage[numbers,sort&compress]{natbib}
\usepackage[colorlinks,linkcolor=blue,citecolor=blue]{hyperref}

\newcommand{\signoff}[1]{}

\numberwithin{equation}{section}
\numberwithin{figure}{section}
\numberwithin{table}{section}
\newtheorem{theorem}{Theorem}[section]
\newtheorem{proposition}[theorem]{Proposition}

\newtheorem{lemma}[theorem]{Lemma}
\theoremstyle{definition}
\newtheorem{definition}[theorem]{Definition}

\newtheorem{remark}[theorem]{Remark}

\newcommand{\bv}[1]{\Vec{#1}}
\newcommand{\bvp}[1]{\mathbf{#1}}
\newcommand{\R}{\mathbb{R}}

\newcommand{\Eh}{E_h}
\newcommand{\Mv}{M_v}
\newcommand{\dt}{\Delta t}
\newcommand{\Xm}{\bv X^{m}}
\newcommand{\Xmp}{\bv X^{m+1}}

\title{A Structure-Preserving Symmetric Galerkin Boundary Element Method for Anisotropic Multi-phase Mullins--Sekerka Flow}

\author{Tokuhiro Eto\thanks{Universit\'{e} Claude Bernard Lyon 1, CNRS, Centrale Lyon, INSA Lyon, Universit\'{e} Jean Monnet, ICJ UMR5208, 69622 Villeurbanne, France. email:eto@math.univ-lyon1.fr}}
\date{\today}

\newcommand{\Bn}{B_n}

\newcommand{\RH}{R_{H}}

\newcommand{\jump}[1]{\llbracket#1\rrbracket}
\newcommand{\curvatureaniso}[0]{\varkappa_\gamma}
\newcommand{\energyaniso}[0]{E_\gamma}
\newcommand{\sgrad}[0]{\nabla_\Gamma}

\newcommand{\dH}[1]{\,\mathrm{d}\mathscr{H}^{#1}}
\newcommand{\dL}[1]{\,\mathrm{d}\mathscr{L}^{#1}}
\newcommand{\dmu}[0]{\,\mathrm{d}\mu}

\newcommand{\ddt}[0]{\frac{\mathrm{d}}{\mathrm{d}t}}
\newcommand{\dde}[0]{\frac{\mathrm{d}}{\mathrm{d}\varepsilon}}
\newcommand{\dtn}[1]{\Lambda_{#1}}
\newcommand{\dtnmulti}[1]{\boldsymbol{\Lambda}_{#1}}
\newcommand{\project}[1]{P_{#1}}
\newcommand{\sol}[1]{S_{#1}}

\newcommand{\phase}[1]{\Omega_{#1}}
\newcommand{\identitymatrix}[1]{\mathrm{I}_{#1}}
\newcommand{\identitymap}[1]{\mathcal{I}_{#1}}
\newcommand{\identitymapbdd}[0]{\mathcal{I}}
\newcommand{\totaln}[0]{N_{\mathrm{tot}}}

\newcommand{\operarorsingle}[1]{\mathcal{V}_{#1}}
\newcommand{\operarordouble}[1]{\mathcal{K}_{#1}}
\newcommand{\operarordoubleregional}[1]{\widetilde{\mathcal{K}}_{#1}}
\newcommand{\operarorhypersingular}[1]{\mathcal{W}_{#1}}
\newcommand{\matrixsingle}[1]{V_{#1}^h}
\newcommand{\matrixdouble}[1]{K_{#1}^h}
\newcommand{\matrixhypersingular}[1]{W_{#1}^h}
\newcommand{\matrixcross}[1]{M_{#1}^h}
\newcommand{\productdual}[3]{\bigl\langle #1,\,#2\bigr\rangle_{#3}}
\newcommand{\productll}[3]{\left(#1,\,#2\right)_{L^2(#3)}}
\newcommand{\productmass}[3]{\left(#1,\,#2\right)_{#3}}

\newcommand{\sgradindex}[1]{\nabla_{\Gamma_{#1}}}

\newcommand{\functionCharacteristic}[1]{\chi_{#1}}

\newcommand{\operatorJ}[1]{\mathcal J_{#1}}

\newcommand{\solGalerkin}[1]{S^h_{#1}}
\newcommand{\solintGalerkin}[1]{S^{h-}_{#1}}
\newcommand{\solextGalerkin}[1]{S^{h+}_{#1}}
\newcommand{\solintextGalerkin}[1]{S^{h\pm}_{#1}}
\newcommand{\projGalerkin}[1]{P^h_{#1}}
\newcommand{\spaceAffineGalerkin}[1]{X^h_{#1}}
\newcommand{\spaceConstGalerkin}[1]{Y^h_{#1}}

\newcommand{\indexphasefrom}[1]{p^-_{#1}}
\newcommand{\indexphaseto}[1]{p^+_{#1}}
\newcommand{\indextj}[2]{s^{#1}_{#2}}

\newcommand{\vectorunit}[1]{\mathbf1_{#1}}
\newcommand{\vectorzero}[1]{\mathbf0_{#1}}
\newcommand{\vectortangent}[1]{\bv \tau_{#1}}

\newcommand{\curvepolygon}[1]{\Gamma^h_{#1}}
\newcommand{\curvepolygonsec}[1]{\Gamma^{#1}}

\newcommand{\junction}[1]{\mathcal{T}_{#1}}
\newcommand{\junctionDisc}[1]{\mathcal{T}^h_{#1}}
\newcommand{\junctionDiscSec}[2]{\mathcal{T}^{#1}_{#2}}

\newcommand{\edges}[2]{\mathcal E_{#1}^{h}}
\newcommand{\edgesDisc}[2]{\mathcal E^{#1}_{#2}}

\newcommand{\region}[1]{\mathcal{R}_{#1}}
\newcommand{\regionDisc}[1]{\mathcal{R}^h_{#1}}

\newcommand{\normal}[0]{\bv{\nu}}
\newcommand{\normaldiscsec}[2]{\bv{\nu}^{#2}_{#1}}
\newcommand{\normalmass}[2]{\bv{\omega}^{#1}_{#2}}

\newcommand{\chemical}[1]{w_{#1}}
\newcommand{\chemicalExtension}[2]{w^{(#1)}_{#2}}
\newcommand{\chemicalvec}[1]{\bvp{w}}
\newcommand{\checmicalDisc}[1]{W_{#1}}
\newcommand{\checmicalDiscVec}[1]{{\bvp{W}}^{#1}}

\newcommand{\femspaceS}[2]{S^{#1}_{#2}}
\newcommand{\femspaceSS}[2]{\mathbb S^{#1}_{#2}}
\newcommand{\femspaceWW}[2]{\mathbb W^{#1}_{#2}}
\newcommand{\femspaceKK}[2]{\mathbb K^{#1}_{#2}}
\newcommand{\femspaceXX}[2]{\mathbb X^{#1}_{#2}}

\newcommand{\areaDisc}[1]{A_{#1}^h}
\newcommand{\areaDiscRegion}[1]{\mathcal A_{#1}^h}

\newcommand{\boundarycomponent}[1]{\Sigma_{#1}}

\newcommand{\matrixl}[1]{\ell^h_{#1}}
\newcommand{\matrixC}[1]{C^h_{#1}}
\newcommand{\matrixD}[1]{D^h_{\Gamma,#1}}
\newcommand{\matrixJ}[1]{J^h_{#1}}

\newcommand{\with}[0]{\qquad\text{with}\quad}
\newcommand{\sentence}[1]{\qquad\text{#1}\quad}

\newcommand{\doi}[1]{\href{https://doi.org/#1}{doi:\nolinkurl{#1}}}

\begin{document}
\maketitle

\begin{abstract}
A structure-preserving symmetric Galerkin boundary element method (SGBEM) is developed for the anisotropic multi-phase Mullins--Sekerka problem in $\R^2$.
The bulk harmonic problems are eliminated through a discrete Dirichlet-to-Neumann operator assembled from region-wise SGBEM discretization,
leading to a boundary-only formulation coupled to the evolution of a polygonal curve network with triple junctions.
The Galerkin approximation of the resulting boundary problem is shown to be uniquely solvable,
and the fully discrete scheme is unconditionally stable with respect to the time step
as long as the evolving polygonal network remains in an admissible class.
The accuracy of the proposed method is demonstrated by a convergence test under coupled space--time refinement against a similarity-rescaled exact solution consisting of three concentric circles.
\end{abstract}


\section{Introduction}\label{sec:intro}
Among curvature-driven free-boundary problems, the Mullins--Sekerka flow is
characterized by a non-local coupling between interfacial geometry and harmonic fields
in the surrounding phases. 
The Gibbs--Thomson relation prescribes the chemical
potential on the interface from its curvature, while the jump of the normal derivative
of the associated harmonic field determines the normal velocity of the interface.
Its numerical approximation therefore involves more than tracking a moving interface:
at each instant, a harmonic problem must be solved throughout the surrounding phases.
The same coupling underlies two fundamental structures of the evolution: the interfacial
energy is dissipated, and the area of each bounded phase is conserved.
In the anisotropic multi-phase setting, the interface is a curve network with curve-wise
direction-dependent surface energies, and the triple junctions must additionally satisfy the
Young--Herring force balance \cite{Herring51}.

The model originates in diffusion-controlled morphological evolution \cite{MS63},
provides a sharp-interface description of late-stage Ostwald ripening \cite{LS61}, and
arises as the sharp-interface limit of the Cahn--Hilliard equation \cite{CH58},
a limit derived and justified in \cite{Pego89,ABC94,S96}.
On the analytical side, existence for the two-phase flow is obtained in the framework
of area-preserving curve-shortening motions \cite{Chen93}, classical well-posedness
with surface tension is established in \cite{EscherSimonett97}, and the multi-phase
formulation goes back to \cite{BGS98}.
More recently, the two-phase Mullins--Sekerka problem in the plane has been recast,
via the method of potentials, as an evolution problem for the interface alone and its
well-posedness established in this boundary-only form \cite{EMM24}, an analytical
counterpart of the viewpoint adopted in the present paper.

Numerical methods for this flow can be grouped according to how they treat the bulk and the interface.
\textit{Phase-field methods} evolve the Cahn--Hilliard equation \cite{CH58} whose singular limit
is the Mullins--Sekerka flow \cite{ABC94};
\textit{level-set and threshold-dynamics methods} developed for geometric interface evolution more broadly
represent the interface implicitly through auxiliary bulk functions \cite{EO15,MBO94,SS11,ZCMO96}.
\textit{Boundary-integral methods} discretize the interface alone and represent the harmonic field by layer potentials
\cite{ZCH96,BCD95,AkaiwaMeiron95,CKT17,BB00}, building on general
boundary-integral and gradient-flow frameworks for interfacial dynamics \cite{HLS94,Mayer00},
with later work reaching multi-component fluids and Ostwald ripening \cite{HLS01} and, through recursively
compressed inverse preconditioning (RCIP),
elliptic problems on boundaries containing triple junctions \cite{HO08,H13}.
\textit{Parametric finite element methods} (PFEM)
\cite{BGN07,BGN08,BGN08a,BGN10,BGN20,BZ21,R23} achieve unconditional stability and structure preservation for anisotropic geometric
evolution equations;
for the Stefan and Mullins--Sekerka problems themselves,
stable PFEM discretization coupling the parametric interface to a bulk finite element field
has been introduced in \cite{BGN10-2}, treating fully anisotropic Gibbs--Thomson laws
with applications to dendritic growth, and extended to one-sided Stefan problems with
approximately crystalline anisotropy and kinetic undercooling in \cite{BGN13},
and recently structure-preserving schemes of this type have been developed
for the multi-phase Mullins--Sekerka problem \cite{EGN24}, its anisotropic variant
with kinetic undercooling \cite{EGN26-2}, and a closely related multi-phase Stefan problem \cite{EGN26}.

These PFEM structure-preserving results are based on bulk or whole-domain weak formulations.
The remaining challenge is to eliminate the bulk degrees of freedom
while retaining a fully discrete energy-dissipation inequality together with the
velocity-level constraints underlying phase-area conservation.

In this paper, the bulk harmonic problems are eliminated through interface
Dirichlet-to-Neumann (Steklov--Poincar\'e) maps,
assembled region by region across the curve network and discretized by the symmetric Galerkin boundary element method (SGBEM)
\cite{Sirtori79,Costabel88,S08,HsiaoWendland08,LangerSteinbach03},
built on the Maue identity \cite{Maue49}.
The variational realization makes the symmetry, positive semi-definiteness, and exact constant kernel of the discrete operator transparent,
and these three properties are used in the proofs of the discrete energy dissipation and velocity-level area conservation.
Boundary operators of this kind have long driven moving-boundary and shape computations in other fields
\cite{HanAtluri02,Frangi02,Dansou19,EpplerHarbrecht06,CraigSulem93},
and energy estimates are standard tools inside boundary-integral analyses \cite{CenicerosHou98,ALS17}.
On the boundary-only side, the author's collocation schemes based on the \textit{charge simulation method} approximate the two-phase and
multi-phase flows, with an area-preserving property in the two-phase case \cite{E24}
and phase-area conservation to machine precision at the velocity level in the multi-phase case \cite{E26};
the present work places this boundary-only approach in a variational Galerkin framework
with unconditional stability and velocity-level area conservation.

The main contributions of this paper are the following.
\begin{itemize}
  \item A boundary-only formulation of the anisotropic multi-phase Mullins--Sekerka
        flow in which the assembled interface Steklov--Poincar\'e operator acts as the
        non-local mobility, together with its SGBEM discretization, proved symmetric
        positive semi-definite with exact constant kernel
        (Proposition~\ref{prop:discrete-dtn-structure}).
  \item A fully discrete semi-implicit scheme for anisotropies of the square-root class
        satisfying an energy-dissipation inequality for every time step
        (Theorem~\ref{thm:main}).
  \item The first variation of the discrete signed area of every bounded phase area vanishes, and the polygonal area
        defect equals exactly the signed area of the displacement polygon
        (Propositions~\ref{prop:cons} and~\ref{prop:areadefect}).
  \item Unique solvability of the gauge-bordered one-step system on
        solvability-admissible curve networks (Theorem~\ref{prop:solv}).
\end{itemize}

Several numerical experiments confirm the dissipation and conservation across three
orders of magnitude in the time step, verify the area-defect identity to machine precision over that sweep, and
exercise the scheme on anisotropic coarsening with a different anisotropy on each phase
and on a junction network mixing three symmetries across its curves.

This paper is organized as follows. In Section~\ref{sec:setting},
we state the continuous multi-phase anisotropic Mullins--Sekerka problem
together with general assumptions for concerned curve network.
In Section~\ref{sec:dtn}, we introduce the interface
Dirichlet-to-Neumann operator together with
anisotropic energy dissipation and area conservation properties of a classical solution.
In Section~\ref{sec:weak}, we present the weak formulation of the target problem reduced to the boundary trace.
In Section~\ref{sec:dio}, we construct the discrete interfacial operator and establish its structure corresponding to the continuous interfacial operator.
In Section~\ref{sec:scheme}, we write out the fully discrete scheme to approximate a solution to the underlying problem
and prove its structural properties.
Section~\ref{sec:numerics} exhibits the numerical experiments to confirm
feasibility and capability for different-anisotropy setting with triple junctions
as well as energy dissipation, area conservation, and the convergence rate
of the discrete flow under a coupled space--time refinement against an isotropic
three-concentric-circle exact solution.
Finally, we summarize the outcomes of this paper and limitations of the present research in Section~\ref{sec:conclusion}.

\section{Problem setting}\label{sec:setting}

Let $\Gamma(t)$ be an evolving curve network in $\R^2$ which is composed of curves
$\Gamma_i(t)\,(1\leq i\leq I_S,\, I_S \geq 1)$ partitioning
the ambient space into connected regions $\region{\ell}(t)\,(1\leq \ell\leq I_R,\, I_R\geq 2)$, exactly one of them unbounded;
for convenience, we suppose that $\region{I_R}(t)$ is unbounded;
each region is associated with one of $I_P\, (\geq 2)$ phases which are indexed by $p = 0,\cdots,I_P -1$.
Each phase is possibly composed of several regions;
to indicate this relationship between phases and regions in the curve network,
we introduce a map $\{0,\ldots,I_P-1\}\ni p\mapsto \phase{p}\in 2^{\{1,\ldots, I_R \}}$,
where $\phase{p}$ denotes the set of all indexes of the regions assigned to the phase $p$;
the unbounded region is always assigned to the phase $0$, that is,
$I_R\in\phase{0}$;
some triplets of three open curves in $\Gamma(t)$ possibly compose triple junction points $\junction{k}(t)\,(1\leq k\leq I_T,\, I_T \geq 0)$;
each three open curves which meet at $\junction{k}$ is identified a sequence of indexes $1\leq \indextj{k}{1} < \indextj{k}{2} < \indextj{k}{3} \leq I_S$;
we set $I_T = 0$ to indicate that $\Gamma(t)$ does not include any triple junction.

In this setting, we aim to approximate a $1$-parametrized family of pairs $\{(\bvp{w}(\cdot,t),\Gamma(t))\}_{t > 0}$ satisfying the following system of equations:
\begin{equation}\label{eq:strong}
  \begin{cases}
    \Delta \bvp{w} = \bvp{0}
      &\quad\text{in}\quad \R^2\setminus\Gamma(t),\ t\in(0,T], \\
    \bvp{w}\cdot\jump{\bvp{\chi}} = \curvatureaniso
      &\quad\text{on}\quad \Gamma(t),\ t\in (0,T], \\
    \jump{\nabla\bvp{w}}\bv{\nu} = -V\jump{\bvp{\chi}}
      &\quad\text{on}\quad \Gamma(t),\ t\in (0,T], \\
    \nabla\bvp{w}(\bv{x},t) = O\!\left(\frac{1}{|\bv{x}|^{2}}\right)
      &\quad\text{as}\quad |\bv{x}|\to\infty,\ t\in (0,T], \\
    \jump{\bvp{w}} = \bvp{0}
      &\quad\text{on}\quad \Gamma(t),\ t\in (0,T], \\
    \bvp{w}\cdot\bvp{1} = 0
      &\quad\text{in}\quad \R^2\setminus\Gamma(t),\ t\in (0,T], \\
    \displaystyle\sum_{\ell=1}^3\varepsilon_{k,\ell}\,
      \bv{\xi}_{\indextj{k}{\ell}}(\bv{\nu}_{\indextj{k}{\ell}})^{\perp} = \bv{0}
      &\quad\text{at}\quad \junction{k}(t),\ 1\leq k \leq I_T,\ t\in (0,T], \\
    \Gamma(0) = \Gamma_0.
  \end{cases}
\end{equation}
Here, the constant $T > 0$ is a time horizon. $\bvp{w}(\cdot,t):\R^2\setminus\Gamma(t)\to\R^{I_P}$ collects the chemical potentials of the
phases, and $\bvp{\chi}(\cdot,t):\R^2\setminus\Gamma(t)\to\{0,1\}^{I_P}$ is the characteristic vector of the
phases.  We suppose that $\Gamma_{i}$ separate different phases $\indexphaseto{i}$ and $\indexphasefrom{i}$,
and the unit normal $\normal_i$ is supposed to point
from $\indexphasefrom{i}$ to $\indexphaseto{i}$, $V_i$ is the normal velocity of $\Gamma_{i}$ in the direction of $\normal_i$,
and the jump of a quantity $u$, which is possibly vector-valued, is defined by
\begin{equation}
  \label{eq:jump}
  \jump{u}(\bv{x}) := \lim_{\varepsilon\,\downarrow\, 0}
    \left\{ u(\bv{x}+\varepsilon\normal_i) - u(\bv{x}-\varepsilon\normal_i) \right\}
  \sentence{for}\bv{x}\in\Gamma_{i},\quad 1\leq i\leq I_S.
\end{equation}
We note that the jump quantity is not defined at the triple junctions $\junction{k}(t)$.

Throughout, for $\bv a=(a_1,a_2)^\top\in\R^2$, we write
$\bv a^{\perp}:=(-a_2,a_1)^\top$ for the counter-clockwise rotation of $\bv a$ by the
angle $\pi/2$.
For each oriented curve, let $\vectortangent{i}$ be the unit tangent chosen so that $\normal_i=-\vectortangent{i}^\perp$.
If the curve is open and $\junction{k}(t)$ is an endpoint of $\Gamma_{\indextj{k}{\ell}}$,
define its endpoint-incidence sign by
\[
 \varepsilon_{k,\ell}:=
 \begin{cases}
  +1,&\junction{k}\text{ is the terminal endpoint of }\Gamma_{\indextj{k}{\ell}},\\
  -1,&\junction{k}\text{ is the initial endpoint of }\Gamma_{\indextj{k}{\ell}}.
 \end{cases}
\]
Then, we observe that $\varepsilon_{k,\ell}\bv\tau_{\indextj{k}{\ell}}$ is the outward unit co-normal of the incident curve at $\junction{k}$,
and the junction condition in \eqref{eq:strong} is independent of the chosen curve parametrization.

Each curve $\Gamma_i$ may carry its own anisotropy $\gamma_i$ of the square-root class:
\begin{equation}\label{eq:gamma}
  \gamma_i(\normal) = c_i \sum_{l=1}^{L_i}
  \sqrt{\normal^{\top}G_{i,l}\normal},
\end{equation}
where $G_{i,l}\in\R^{2\times2}$ is symmetric positive definite and $c_i > 0$ for $1\leq i\leq I_S$.

Each $\gamma_i$ is convex and positively one-homogeneous
(see \cite[Eq.~(2.15)]{BGN10}).  On $\Gamma_i$,
the Cahn--Hoffman vector and anisotropic curvature are defined by
\[
  \bv\xi_i(\bv\nu_i):=\nabla\gamma_i(\bv\nu_i),
  \qquad
  \kappa_{\gamma_i}:=-\sgrad\cdot\bv\xi_i(\bv\nu_i),
\]
where $\sgrad$ denotes the (vector-valued) surface gradient, and $\curvatureaniso$ denotes
the function on $\Gamma$ whose restriction to each $\Gamma_i$ is $\kappa_{\gamma_i}$.
Since each $\Gamma_i$ is a curve, the surface gradient of a scalar function $\phi$ is
determined by the scalar tangential derivative $\partial_{s_i}\phi$ with respect to the
arc-length of $\Gamma_i$ through
\[
  \sgradindex{i}\phi=(\partial_{s_i}\phi)\,\bv\tau_i;
\]
we reserve $\sgrad$ for the vector-valued operator, and write $\partial_s$ for the scalar
tangential derivative.
The anisotropic surface energy is defined by
\begin{equation}
  \label{eq:energy}
E_{\gamma}(\Gamma)
:=\sum_{i=1}^{I_S}\int_{\Gamma_i}\gamma_i(\bv\nu_i)\,\dH{1},
\end{equation}
where $\mathscr{H}^1$ denotes the $1$-dimensional Hausdorff measure.
In particular, if $\gamma_i(\normal)=\sigma_i|\normal|$, then
$E_{\gamma}(\Gamma)=\sum_i\sigma_i\mathscr{H}^1(\Gamma_i)$ holds; for equal
$\sigma_i$, the Young--Herring condition in \eqref{eq:strong} reduces to the $2\pi/3$ angle condition.
Finally, $\Gamma_0$ is given as an initial curve network.

The junction condition in \eqref{eq:strong} is written in the contracted planar form.
To relate it to the formulation in terms of normals and co-normals, write
$\bv\mu_{k,\ell}:=\varepsilon_{k,\ell}\bv\tau_{\indextj{k}{\ell}}$ for the outward unit
co-normal introduced above.  Since $\gamma_i$ is positively homogeneous degree one,
Euler's relation gives $\bv\xi_i(\bv\nu_i)\cdot\bv\nu_i=\gamma_i(\bv\nu_i)$, so that
$\bv\xi_i(\bv\nu_i)=\gamma_i(\bv\nu_i)\bv\nu_i
+\bigl(\bv\xi_i(\bv\nu_i)\cdot\bv\tau_i\bigr)\bv\tau_i$;
applying the rotation $\perp$ and using $\bv\nu_i^{\perp}=\bv\tau_i$ and $\bv\tau_i^{\perp}=-\bv\nu_i$,
we obtain
\begin{equation}\label{eq:herring-conormal}
 \varepsilon_{k,\ell}\,\bv\xi_{\indextj{k}{\ell}}(\bv\nu_{\indextj{k}{\ell}})^{\perp}
 =\gamma_{\indextj{k}{\ell}}(\bv\nu_{\indextj{k}{\ell}})\,\bv\mu_{k,\ell}
 -\bigl(\bv\xi_{\indextj{k}{\ell}}(\bv\nu_{\indextj{k}{\ell}})\cdot\bv\mu_{k,\ell}\bigr)
  \bv\nu_{\indextj{k}{\ell}}.
\end{equation}
Summing \eqref{eq:herring-conormal} over $\ell=1,\,2,\,3$ shows that the junction condition in
\eqref{eq:strong} coincides with the anisotropic force balance of \cite[Eq.~(1e)]{EGN26-2}.

\section{Interface operator}\label{sec:dtn}
In this section, we introduce an interfacial operator that describes the target
problem \eqref{eq:strong} as a boundary-only problem.
Let $H^{\frac12}(\Gamma)$ be the Hilbert space of compatible traces on the curve network $\Gamma$,
equipped with the quotient trace norm.  For a region $\region{}$, let
\[
  \project{\region{}}:H^{\frac12}(\Gamma)
  \longrightarrow H^{\frac12}(\partial\region{})
\]
be restriction to the boundary of $\region{}$.

Given $\phi\in H^{\frac12}(\partial\region{})$, let $u$ be its harmonic extension, that is,
\begin{equation}\label{eq:laplace-dirichlet}
  \begin{cases}
  \Delta u=0 &\sentence{in}\region{},\\
  u=\phi &\sentence{on}\partial\region{},
  \end{cases}
\end{equation}
with the far-field condition in \eqref{eq:strong} when $\region{}$ is unbounded.

Let $\sol{\region{}}:H^{\frac12}(\partial\region{})\to H^{-\frac12}(\partial\region{})$ be the
Dirichlet-to-Neumann map defined by
\[
  \sol{\region{}}\phi := \frac{\partial u}{\partial \normal_{\region{}}},
\]
where $\normal_{\region{}}$ is the outward unit normal vector on $\partial\region{}$.

The scalar interface Dirichlet-to-Neumann operator is defined as the assembled operator:
\begin{equation}\label{eq:dtn}
  \dtn{\Gamma}
  :=\sum_{r=1}^{I_R}
  \project{\region{r}}^*\,\sol{\region{r}}\,\project{\region{r}}
  :H^{\frac12}(\Gamma)\longrightarrow H^{-\frac12}(\Gamma),
\end{equation}
where for a region $\region{}$, $\project{\region{}}^*:H^{-\frac12}(\partial\region{})\to H^{-\frac12}(\Gamma)$ assembles the regional flux on $\partial\region{}$ into a distribution on $\Gamma$.
The component-wise action by $\dtn{\Gamma}$ on the trace
$\bvp{\phi}=(\phi_0,\ldots,\phi_{I_P-1})^\top$ of the multi-phase chemical potential is represented as
\[
  \dtnmulti{\Gamma}
  :=\identitymatrix{I_P}\otimes\dtn{\Gamma},
\]
where $\identitymatrix{I_P}$ denotes the identity matrix in $\R^{I_P\times I_P}$;
for a matrix $A$ and a linear map $T$, $A\otimes T$ denotes the Kronecker product, which represents the component-wise action.

According to the above definition, the operator $\dtnmulti{\Gamma}$ maps the chemical-potential traces to their assembled normal fluxes.
For later use, we write the dual product of $v\in H^{-\frac12}(\Gamma)$ and $w\in H^{\frac12}(\Gamma)$ as $\productdual{v}{w}{\Gamma}$,
while we use the notation that $\productll{v}{w}{\Gamma}$ for $v,\,w\in L^2(\Gamma)$ to denote the standard $L^2$-inner product.

We now present important properties of the interfacial operator $\dtn{\Gamma}$.
\begin{proposition}\label{prop:dtn-structure}
The operator $\dtn{\Gamma}$ is symmetric, that is,
  \[
  \productdual{\dtn{\Gamma}v}{w}{\Gamma} = \productdual{\dtn{\Gamma}w}{v}{\Gamma} \qquad\text{for every}\quad v,\,w\in H^{\frac12}(\Gamma).
  \]
  Moreover, $\dtn{\Gamma}$ is positive semi-definite, that is,
  \[
  \productdual{\dtn{\Gamma}w}{w}{\Gamma}\geq 0\qquad \text{for every}\quad w\in H^{\frac12}(\Gamma).
  \]
  Finally, the kernel of the operator $\dtn{\Gamma}$ consists of all constants, that is,
  \[
  \ker\dtn{\Gamma}=\operatorname{span}\{1\}.
  \]
\end{proposition}
\begin{proof}
For $\chemical{1},\,\chemical{2}\in H^{\frac12}(\Gamma)$,
let $\chemicalExtension{r}{1}$ and $\chemicalExtension{r}{2}$ be the
harmonic extensions of $\project{\region{r}}\chemical{1}$ and
$\project{\region{r}}\chemical{2}$, respectively.
Green's identity gives
\begin{equation}
  \label{eq:green-identity-1}
\productdual{\sol{\region{r}}\project{\region{r}}\chemical{2}}{\project{\region{r}}\chemical{1}}{\partial\region{r}}
= \int_{\region{r}}\nabla\chemicalExtension{r}{1}\cdot\nabla \chemicalExtension{r}{2}\,\dL{2}
\end{equation}
for $1\leq r\leq I_R-1$, where $\mathscr{L}^2$ denotes the $2$-dimensional Lebesgue measure.

For the unbounded region
$\region{I_R}$, we choose $L > 0$ so large that $\Gamma\subset B_L(\bv 0)$ and
apply it on $\region{I_R}\cap B_L(\bv 0)$ to derive
\begin{equation}
  \label{eq:green-identity-2}
  \productdual{\sol{\region{I_R}}\project{\region{I_R}}\chemical{2}}
    {\project{\region{I_R}}\chemical{1}}{\partial\region{I_R}}
  =\int_{\region{I_R}\cap B_L(\bv 0)}
    \nabla \chemicalExtension{I_R}{1}\cdot\nabla \chemicalExtension{I_R}{2}\,\dL{2}
  -\int_{\partial B_L(\bv 0)}\chemicalExtension{I_R}{1}\frac{\partial \chemicalExtension{I_R}{2}}{\partial \normal_{B_L(\bv 0)}}\,\dH{1},
\end{equation}
where $\normal_{B_L(\bv{0})}(\bv x) = \bv x / |\bv x|$ for $\bv x\neq \bv 0$.
The far-field condition in \eqref{eq:strong} implies that $\chemicalExtension{I_R}{1}$ is bounded (see \cite[Lemma 3]{E24}), and
\[
\frac{\partial\chemicalExtension{I_R}{2}}{\partial\normal_{B_L(\bv 0)}} = O\left(\frac{1}{L^2}\right)\qquad\text{on}\quad \partial B_L(\bv 0)\qquad\text{as}\quad L\to\infty,
\]
and hence, the boundary integral is $O(L^{-1})$ as $L\to\infty$.
Letting $L\to\infty$ and summing \eqref{eq:green-identity-1} and \eqref{eq:green-identity-2} over $1\leq r\leq I_R$
yields
\begin{equation}
  \label{eq:energy-indentity}
  \productdual{\dtn{\Gamma}\chemical{2}}{\chemical{1}}{\Gamma}
  =\sum_{r=1}^{I_R}\int_{\region{r}}
    \nabla \chemicalExtension{r}{1}\cdot\nabla \chemicalExtension{r}{2}\,\dL{2}.
\end{equation}
The right-hand side is symmetric with respect to $\chemical{1}$ and $\chemical{2}$ and is non-negative when $\chemical{1} = \chemical{2}$.
  If $\chemical{}\equiv c$ is constant, every regional harmonic extension of $\chemical{}$ is constant,
  and thus, we have $\dtn{\Gamma}c=0$ which means $\chemical{}\in\ker\dtn{\Gamma}$.

  Conversely, if $\chemical{}\in\ker\dtn{\Gamma}$, the energy identity \eqref{eq:energy-indentity} with
$\chemical{1}=\chemical{2} = w$ gives $\nabla \chemicalExtension{r}{}=0$ for every $1\leq r\leq I_R$.
Since each region $\region{r}$ is connected, $\chemicalExtension{r}{}$ is constant.
Two regions sharing a curve of $\Gamma$ carry the same constant, since the
traces of their harmonic extensions on that curve coincide.
Any two regions are joined by a chain of such pairs.
Hence, all the constants coincide and $\chemical{}\equiv c$ for some $c\in\R$.
The proof is now complete.
\end{proof}

As an example, we confirm that the $I_P$-potential formulation
\eqref{eq:strong} collapses, for two phases, to the classical Mullins--Sekerka law driven
by a regional Dirichlet-to-Neumann map.
Suppose that $I_R = I_P = 2$, $I_S = 1$, $I_T = 0$, and $\normal_1$ is the unit normal vector on $\Gamma_1$ pointing into the region $\region{1}$;
the relationship between the regions and phases is described as
$\phase{0} = \{2\}$, and $\phase{1} = \{1\}$.
In this setting, the system \eqref{eq:strong} is reduced to
\begin{equation*}
  \begin{cases}
    \Delta \chemical{0}(\cdot,t) = 0\qquad&\text{in}\quad\R^2\setminus\Gamma_1(t),\\
    \Delta \chemical{1}(\cdot,t) = 0\qquad&\text{in}\quad\R^2\setminus\Gamma_1(t),\\
    \chemical{1}(\cdot,t) - \chemical{0}(\cdot,t) = \curvatureaniso\qquad&\text{on}\quad\Gamma_1(t),\\
    -V\begin{pmatrix}
      -1\\1
    \end{pmatrix} = \begin{pmatrix}
      \jump{\nabla\chemical{0}}\normal_1\\
      \jump{\nabla\chemical{1}}\normal_1
    \end{pmatrix}\qquad &\text{on}\quad\Gamma_1(t).
  \end{cases}
\end{equation*}
By the zero-sum condition $\chemical{1} + \chemical{0} = \bvp{w}\cdot\bvp{1} = 0$, we have the system with respect to $\chemical{1}$:
\begin{equation*}
  \begin{cases}
    \Delta\chemical{1} = 0\qquad & \text{in}\quad\R^2\setminus\Gamma_{1}(t),\\
    \chemical{1} = \frac12 \curvatureaniso & \text{on}\quad\Gamma_1(t),\\
    V = -\jump{\nabla\chemical{1}}\cdot\normal_1\qquad & \text{on}\quad\Gamma_1(t).
  \end{cases}
\end{equation*}
Thus, we have
\begin{equation*}
  \begin{aligned}
  V &= -\left(\nabla\chemical{1}\lfloor_{\region{1}} - \nabla\chemical{1}\lfloor_{\region{2}}\right)\cdot\normal_1
  = \nabla\chemical{1}\lfloor_{\region{1}}\cdot\normal_{\region{1}}
    +\nabla\chemical{1}\lfloor_{\region{2}}\cdot\normal_{\region{2}}\\
  &=\sol{\region{1}}\left(\frac12\curvatureaniso\right)
    +\sol{\region{2}}\left(\frac12\curvatureaniso\right)
    =\frac12\dtn{\Gamma}\curvatureaniso,
  \end{aligned}
\end{equation*}
where we note that $\project{\region{1}}$ and $\project{\region{2}}$ are identical since $\Gamma_1 = \partial\region{1} = \partial\region{2}$.
The zero-sum gauge condition splits the anisotropic curvature equally between $\chemical{0}$ and
$\chemical{1}$, and the vector system reduces to this scalar law.

In the general multi-phase case, the interface
operator \eqref{eq:dtn} satisfies

\begin{equation}
  \label{eq:mjump}
\dtn{\Gamma}\chemical{p} = -\jump{\nabla\chemical{p}}\cdot\normal,
\end{equation}
where $\chemical{p}$ is understood as the trace of $\chemical{p}$ on $\Gamma(t)$.
Thus, the motion law of \eqref{eq:strong} reads
\begin{equation}\label{eq:mstefan}
  \dtn{\Gamma}\chemical{p} = V\jump{\bvp{\chi}_p}
  \qquad\text{on}\quad\Gamma(t),\qquad p = 0,\cdots,I_P-1.
\end{equation}
Since $\dtn{\Gamma}$ is symmetric with $\dtn{\Gamma}1=0$ according to Proposition~\ref{prop:dtn-structure},
every solution of \eqref{eq:mstefan} satisfies
\[
\productdual{V\jump{\bvp{\chi}_p}}{1}{\Gamma}
=\productdual{\dtn{\Gamma}\chemical{p}}{1}{\Gamma}
=\productdual{\dtn{\Gamma}1}{\chemical{p}}{\Gamma}=0,
\]
and thus, its right-hand side annihilates the constants,
\begin{equation}
  \label{eq:stefan-compatibility}
\int_\Gamma \jump{\bvp{\chi}_p} V\,\dH{1} = 0\qquad\text{for every}\quad p = 0,\cdots,I_P-1.
\end{equation}
For $1\leq p\leq I_P-1$, this compatibility condition is the per-phase area
conservation of Proposition~\ref{prop:ccons}.

In general, for $s\in\R$, write
\begin{equation}\label{eq:arcwise-spaces}
\mathcal{X}_{\Gamma}^s:=\prod_{i=1}^{I_S}H^s(\Gamma_i)
,\qquad \mathcal X_\Gamma^0=\prod_{i=1}^{I_S}L^2(\Gamma_i),
\end{equation}
so that no compatibility at the junctions is imposed on the elements of
$\mathcal X_\Gamma^s$, and $\mathcal X_\Gamma^{-\frac12}$ is the dual of
$\mathcal X_\Gamma^{\frac12}$ with respect to the arc-wise pairing
$\productdual{v}{w}{\Gamma}=\sum_{i=1}^{I_S}\productdual{v_i}{w_i}{\Gamma_i}$, which
extends the pairing of $H^{-\frac12}(\Gamma)$ with $H^{\frac12}(\Gamma)$ used above.
To pass between compatible traces and arc-wise data, we use the curve-wise restriction
\[
  \project{\Gamma}:H^{\frac12}(\Gamma)\longrightarrow\mathcal X_\Gamma^{\frac12},
  \qquad
  \project{\Gamma}\phi:=(\phi\lfloor_{\Gamma_1},\ldots,\phi\lfloor_{\Gamma_{I_S}})^\top,
\]
which is bounded and injective, and its adjoint
$\project{\Gamma}^{*}:\mathcal X_\Gamma^{-\frac12}\to H^{-\frac12}(\Gamma)$, which assembles
arc-wise densities into a functional on compatible traces.  These play for the curves
$\Gamma_i$ the role that $\project{\region{}}$ and $\project{\region{}}^*$ play for the
region boundaries $\partial\region{}$; the discrete counterpart of $\project{\Gamma}$ is
the curve-restriction matrix \eqref{eq:curve-restriction} below.
For interface trace data
$\phi=(\phi_1,\ldots,\phi_{I_S})^\top\in\mathcal{X}_{\Gamma}^s$,
define the phase-incidence multiplication operator:
\begin{equation}\label{eq:Mp}
  M_p\phi:=
  \bigl(m_{p,1}\phi_1,\ldots,m_{p,I_S}\phi_{I_S}\bigr)^\top
  \qquad\text{with}\quad
  m_{p,i}:=\jump{\bvp{\chi}_p}\big|_{\Gamma_i}\in\{-1,0,1\},
\end{equation}
and in operator-block notation,
\[
  M_p=\operatorname{diag}\bigl(m_{p,1}\identitymap{1},\ldots,
  m_{p,I_S}\identitymap{I_S}\bigr).
\]
Here, $\identitymap{} = (\identitymap{i})_{i=1}^{I_S}$ denote the identity embeddings, that is,
$\mathcal I_i(\bv x) := \bv x$ for each $\bv x\in\Gamma_i$.
The same notation is used on $\mathcal{X}_{\Gamma}^{1/2}$ and its dual
$\mathcal{X}_{\Gamma}^{-1/2}$.  Thus, $M_p$ has $I_S\times I_S$ operator blocks; the phase label $p$ selects one such operator.
The sign $m_{p,i}$ is zero when phase $p$ is not adjacent to $\Gamma_i$, and otherwise
records on which side of the oriented interface that phase lies.

As an example of the operators $M_p$, consider the three-phase double bubble shown in
Figure~\ref{fig:three-phase-double-bubble}.
Phase $0$ occupies the exterior, phases $1$ and $2$ respectively occupy
the left and right domains, and
\[
\Gamma=\Gamma_{1}\cup\Gamma_{2}\cup\Gamma_{3},
\]
with the interfaces oriented from $0$ to $1$, from $1$ to $2$, and from $2$ to $0$,
respectively.
\begin{figure}[t]
  \centering
  \begin{tikzpicture}[
      scale=0.82,
      >=stealth,
      interface/.style={draw=blue!25!black, line width=1.15pt},
      normal/.style={->, draw=red!75!black, line width=0.9pt},
      phase/.style={font=\large\bfseries, text=blue!20!black},
      curve label/.style={font=\small, fill=white, fill opacity=0.86,
                          text opacity=1, inner sep=2.5pt, rounded corners=1.5pt}
    ]
    \coordinate (T) at (0,1.48);
    \coordinate (B) at (0,-1.48);

    \path[fill=gray!14] (-3.65,-2.55) rectangle (3.65,2.70);

    \path[fill={rgb,255:red,86;green,180;blue,233}, fill opacity=0.28]
      (T) .. controls (-1.20,2.18) and (-2.42,1.16) .. (-2.42,0)
          .. controls (-2.42,-1.16) and (-1.20,-2.18) .. (B) -- cycle;
    \path[fill={rgb,255:red,230;green,159;blue,0}, fill opacity=0.28]
      (T) .. controls (1.20,2.18) and (2.42,1.16) .. (2.42,0)
          .. controls (2.42,-1.16) and (1.20,-2.18) .. (B) -- cycle;

    \draw[interface] (T) .. controls (-1.20,2.18) and (-2.42,1.16) .. (-2.42,0)
                             .. controls (-2.42,-1.16) and (-1.20,-2.18) .. (B);
    \draw[interface] (T) -- (B);
    \draw[interface] (B) .. controls (1.20,-2.18) and (2.42,-1.16) .. (2.42,0)
                             .. controls (2.42,1.16) and (1.20,2.18) .. (T);

    \fill[blue!25!black] (T) circle (2.1pt) (B) circle (2.1pt);
    \node[phase] at (-1.02,-0.66) {$1$};
    \node[phase] at (1.02,-0.66) {$2$};
    \node[phase] at (0,2.34) {$0$};
    \node[curve label] at (-1.62,0.78) {$\Gamma_{1}$};
    \node[curve label] at (0.38,0.72) {$\Gamma_{2}$};
    \node[curve label] at (1.62,0.78) {$\Gamma_{3}$};

    \draw[normal] (-2.42,0) -- (-1.58,0)
      node[midway,below=2pt,text=red!55!black] {$\normal_{1}$};
    \draw[normal] (0,0) -- (0.84,0)
      node[midway,below=2pt,text=red!55!black] {$\normal_{2}$};
    \draw[normal] (2.42,0) -- (3.26,0)
      node[midway,below=2pt,text=red!55!black] {$\normal_{3}$};
  \end{tikzpicture}

  \caption{Three-phase double bubble. The arrows show the normals defining the jumps.}
  \label{fig:three-phase-double-bubble}
\end{figure}
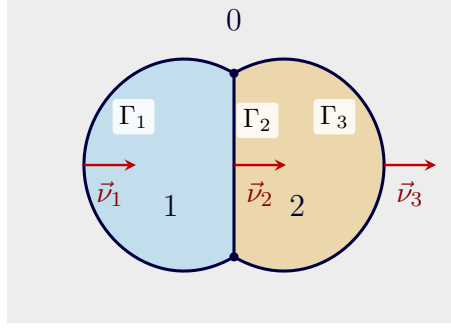
Ordering the interfaces by $(\Gamma_1,\Gamma_2,\Gamma_3)$, definition~\eqref{eq:Mp}
gives
\begin{equation*}
\begin{aligned}
M_0&=\operatorname{diag}(-\identitymap{1},0,\identitymap{3}),&
M_1&=\operatorname{diag}(\identitymap{1},-\identitymap{2},0),&
M_2&=\operatorname{diag}(0,\identitymap{2},-\identitymap{3}).
\end{aligned}
\end{equation*}

As a conclusion of this section, we shall show two important properties of a classical solution to \eqref{eq:strong},
a dissipation property of the anisotropic interfacial energy $\energyaniso$ and an area-preserving property of each phase.
Though they can be shown in same ways as in \cite[Proposition~4.1 and 4.2]{E26}, we reproduce their proofs for completeness in terms of
the introduced interfacial operator $\dtn{\Gamma}$.

\begin{proposition}[Energy dissipation]\label{prop:diss}
  Let $\{(\bvp{w}(\cdot,t),\Gamma(t))\}_{t > 0}$ be a smooth solution to \eqref{eq:strong}.
Then, the interfacial energy \eqref{eq:energy} is non-increasing in time.
Precisely, it holds that
\begin{equation}\label{eq:cdiss}
  \ddt\energyaniso(\Gamma(t))
  = -\sum_{p = 0}^{I_P-1}\productdual{\dtn{\Gamma(t)}\chemical{p}}{\chemical{p}}{\Gamma(t)} \leq 0.
\end{equation}
\end{proposition}
\begin{proof}
Using \eqref{eq:mjump}, and the motion law and Gibbs--Thomson law in
\eqref{eq:strong}, we obtain
\begin{equation*}
  \begin{aligned}
\sum_{p = 0}^{I_P-1}\productdual{\dtn{\Gamma(t)}\chemical{p}}{\chemical{p}}{\Gamma(t)}
&= \sum_{p=0}^{I_P-1}
   \productdual{-\jump{\nabla\chemical{p}}\cdot\normal}{\chemical{p}}{\Gamma(t)}
 = \sum_{p=0}^{I_P-1}\productdual{V\jump{\bvp{\chi}_p}}{\chemical{p}}{\Gamma(t)}\\
&= \productdual{V}{\bvp{\chemical{}}\cdot\jump{\bvp{\chi}}}{\Gamma(t)}
 = \productdual{V}{\curvatureaniso}{\Gamma(t)}.
  \end{aligned}
\end{equation*}
Meanwhile, we recall from \cite[Lemma 3.1 and Eq.~(3.4)]{BGN10} and \cite[Eq.~(6)]{EGN26-2} that
\[
\ddt\energyaniso(\Gamma(t)) = -\sum_{i=1}^{I_S}\int_{\Gamma_i(t)}V\curvatureaniso\dH{1} = -\productdual{V}{\curvatureaniso}{\Gamma(t)}.
\]
Combining the above two identities concludes the proof.
\end{proof}

\begin{proposition}[Area conservation per phase]\label{prop:ccons}
  Assume that $\{(\bvp{w}(\cdot,t),\Gamma(t))\}_{0<t\leq T}$ is a smooth solution to \eqref{eq:strong}.
  For each $p = 0,\cdots,I_P-1$, let $A_p(t)$ be the area enclosed by the regions associated with the phase $p$ at time $t$, that is,
  \[
  A_p(t) := \sum_{\substack{1\leq r\leq I_R\\r\in\phase{p}}}\left|\region{r}(t)\right|.
  \]
Then, it holds that
\[
\ddt A_p(t) = 0\qquad\text{for all}\quad p = 1,\cdots,I_P-1.
\]
Here, we have excluded the case of the phase $0$ from the statement since $A_0(t) = \infty$.
\end{proposition}
\begin{proof}
  This can be shown by a same way as \cite[Proposition 4.2]{E26}.
\end{proof}

\section{Boundary-reduced weak formulation}\label{sec:weak}

We now present a boundary-reduced weak formulation of \eqref{eq:strong}.
To this end, define the following boundary element spaces:
\begin{equation}
  \label{eq:weak-bes-1}
 \mathbb S_\Gamma:=H^{\frac12}(\Gamma),
 \qquad
 \mathbb S_\Gamma^{I_P}:=[H^{\frac12}(\Gamma)]^{I_P},
 \qquad
 \mathbb K_\Gamma:=\mathcal X_\Gamma^{0},
\end{equation}
and define the junction-compatible geometric space:
\begin{equation}
  \label{eq:weak-bes-2}
 \mathbb X_\Gamma
 :=\left\{\bv\eta=(\bv\eta_i)_{i=1}^{I_S}
 \in\prod_{i=1}^{I_S}H^1(\Gamma_i;\R^2)\biggm|
 \bv\eta_{\indextj{k}{1}}(\junction{k})
 =\bv\eta_{\indextj{k}{2}}(\junction{k})
 =\bv\eta_{\indextj{k}{3}}(\junction{k}),\quad 1\leq \forall k\leq I_T
 \right\}.
\end{equation}
The phase--interface incidence operator and its adjoint are
\[
 \begin{aligned}
 \mathcal D_\Gamma^*:\mathbb S_\Gamma^{I_P}&\longrightarrow\mathcal X_\Gamma^{\frac12},
 &\mathcal D_\Gamma^*\bvp\phi&:=\sum_{p=0}^{I_P-1}M_p\project{\Gamma}\phi_p,\\
 \mathcal D_\Gamma:\mathcal X_\Gamma^{-\frac12}&\longrightarrow
 [H^{-\frac12}(\Gamma)]^{I_P},
 &(\mathcal D_\Gamma V)_p&:=\project{\Gamma}^{*}M_pV,
 \end{aligned}
\]
where
\[
 \sum_{p=0}^{I_P-1}
 \productdual{(\mathcal D_\Gamma V)_p}{\phi_p}{\Gamma}
 =\productdual{V}{\mathcal D_\Gamma^*\bvp\phi}{\Gamma}.
\]
Here, the left-hand pairing is that of $H^{-\frac12}(\Gamma)$ with $H^{\frac12}(\Gamma)$ and
the right-hand one is the arc-wise pairing of \eqref{eq:arcwise-spaces}, so that the
identity is exactly the statement that $\mathcal D_\Gamma$ is the adjoint of
$\mathcal D_\Gamma^*$.

Since $\ker\dtn{\Gamma}=\operatorname{span}\{1\}$ (see Proposition~\ref{prop:dtn-structure}),
fix the remaining common additive
constant by
\[
 \mathbb W_\Gamma
 :=\left\{\bvp w\in\mathbb S_\Gamma^{I_P}:
 \mathfrak g_\Gamma(\bvp w)=0\right\}
 \qquad\text{with}\quad
 \mathfrak g_\Gamma(\bvp w)
 :=\sum_{p=0}^{I_P-1}\int_\Gamma w_p\,\dH{1}.
\]
Define
\begin{equation}\label{eq:continuous-anisotropic-form}
 a_{\boldsymbol\gamma,\Gamma}(\mathcal I,\bv\eta)
 :=\sum_{i=1}^{I_S}\int_{\Gamma_i}
 \bv\xi_i(\bv\nu_i)^\perp\cdot\partial_{s_i}\bv\eta_i\,\dH{1},
\end{equation}
where $\partial_{s_i}$ is taken in the direction of $\bv\tau_i$ fixed above.
Under these notations,
the boundary-reduced mixed problem corresponding to \eqref{eq:strong} is to seek
\[
 (\bvp w,\kappa,V)
 \in\mathbb W_\Gamma\times\mathbb K_\Gamma\times\mathcal X_\Gamma^{-\frac12}
\]
such that
\begin{align}
 \sum_{p=0}^{I_P-1}
 \productdual{\dtn{\Gamma}w_p}{\phi_p}{\Gamma}
 -\productdual{V}{\mathcal D_\Gamma^*\bvp\phi}{\Gamma}
 &=0
 &&\forall\bvp\phi\in\mathbb S_\Gamma^{I_P},
 \label{eq:continuous-reduced-field}\\
 \productll{\mathcal D_\Gamma^*\bvp w-\kappa}{\psi}{\Gamma}&=0
 &&\forall\psi\in\mathbb K_\Gamma,
 \label{eq:continuous-reduced-gibbs}\\
 \sum_{i=1}^{I_S}\int_{\Gamma_i}
 \kappa_i\bv\nu_i\cdot\bv\eta_i\,\dH{1}
 +a_{\boldsymbol\gamma,\Gamma}(\mathcal I,\bv\eta)&=0
 &&\forall\bv\eta\in\mathbb X_\Gamma.
 \label{eq:continuous-reduced-curvature}
\end{align}
The first two equations follow by testing the trace equations \eqref{eq:mstefan}
against $\phi_p\in\mathbb S_\Gamma$
in the duality pairing $\productdual{\cdot}{\cdot}{\Gamma}$, and
the Gibbs--Thomson law in \eqref{eq:strong} against $\psi\in\mathbb K_\Gamma$ in $L^2(\Gamma)$.
The mean gauge defining $\mathbb W_\Gamma$ also recovers the pointwise
zero-sum gauge in \eqref{eq:strong}: choosing
$\bvp\phi=(\phi,\ldots,\phi)$ in the first equation and using
$\sum_pM_p=0$ yields
$\dtn{\Gamma}(\sum_pw_p)=0$.  Thus,
$\sum_pw_p=c$ by $\ker\dtn{\Gamma}=\operatorname{span}\{1\}$, and
$\mathfrak g_\Gamma(\bvp w)=0$ implies $c=0$.
To show \eqref{eq:continuous-reduced-curvature}, we recall the continuous weak
curvature identity of \cite[Lemma~3.1 and Eq.~(3.4)]{BGN10} and
\cite[Eq.~(17c)]{EGN26-2}.  For any $\bv\eta\in\mathbb X_\Gamma$,
curvewise integration by parts and
$\kappa_{\gamma_i}=-\sgradindex{i}\cdot\bv\xi_i(\bv\nu_i)$ give
\begin{align*}
 \sum_{i=1}^{I_S}\int_{\Gamma_i(t)}
   \kappa_{\gamma_i}\bv\nu_i\cdot\bv\eta_i\,\dH{1}
 +a_{\boldsymbol\gamma,\Gamma(t)}(\mathcal I,\bv\eta)
 &=\sum_{k=1}^{I_T}\sum_{\ell=1}^{3}
   \varepsilon_{k,\ell}\,
   \bv\eta_{\indextj{k}{\ell}}(\junction{k})\cdot
   \bv\xi_{\indextj{k}{\ell}}
   \bigl(\bv\nu_{\indextj{k}{\ell}}\bigr)^\perp\\
 &=\sum_{k=1}^{I_T}\bv\eta(\junction{k})\cdot
   \sum_{\ell=1}^{3}\varepsilon_{k,\ell}\,
   \bv\xi_{\indextj{k}{\ell}}
   \bigl(\bv\nu_{\indextj{k}{\ell}}\bigr)^\perp
 =0.
\end{align*}
The last equality follows from the junction compatibility of $\bv\eta$ and the
oriented Young--Herring law in \eqref{eq:strong}.

\section{Discrete interfacial operator}
\label{sec:dio}

We first give the continuous boundary-integral representation.  For a region $\region{}$,
set $\Gamma_{\region{}}:=\partial \region{}$ and let $\bv\nu_{\region{}}$ be its outward unit normal.
We use this normal in both the boundary-integral operators and the regional
Dirichlet-to-Neumann map.  In particular, for the unbounded region,
$\bv\nu_{\region{}}$ points into its bounded complement.
Let $G$ be the fundamental solution of the planar Laplacian defined by
\[
G(\bv{x})=-\frac{1}{2\pi}\log|\bv{x}|\qquad\text{for}\quad \bv{x}\neq \bv{0},
\]
and define the single-layer, double-layer, and hypersingular operators
$\operarorsingle{\region{}}:H^{-\frac12}(\Gamma_{\region{}})\to H^{\frac12}(\Gamma_{\region{}})$, $\operarordouble{\region{}}:H^{\frac12}(\Gamma_{\region{}})\to H^{\frac12}(\Gamma_{\region{}})$, and $\operarorhypersingular{\region{}}:H^{\frac12}(\Gamma_{\region{}})\to H^{-\frac12}(\Gamma_{\region{}})$ on $\Gamma_{\region{}}$ by
\begin{equation}\label{eq:bie-operators}
\begin{aligned}
  (\operarorsingle{\region{}}\psi)(\bv{x})
    &:=\int_{\Gamma_{\region{}}}G(\bv{x}-\bv{y})\psi(\bv{y})\,\dH{1}(\bv{y})&&\sentence{for}\psi\in H^{-\frac12}(\Gamma_{\region{}}),\\
  (\operarordouble{\region{}}\phi)(\bv{x})
    &:=\int_{\Gamma_{\region{}}}\frac{\partial G}{\partial \bv\nu_{\region{}}(\bv y)}(\bv{x}-\bv{y})\phi(\bv{y})\,\dH{1}(\bv{y})&&\sentence{for}\phi\in H^{\frac12}(\Gamma_{\region{}}),\\
  (\operarorhypersingular{\region{}}\phi)(\bv{x})&:=-\frac{\partial \operarordoubleregional{\region{}}}{\partial \bv\nu_{\region{}}(\bv x)}\phi(\bv{x})&&\sentence{for}\phi\in H^{\frac12}(\Gamma_{\region{}}).
\end{aligned}
\end{equation}
Here, $\operarordoubleregional{\region{}}$ denotes the double-layer potential given by
the same integral as $\operarordouble{\region{}}$ for $\bv x\notin\Gamma_{\region{}}$;
the normal derivative defining $\operarorhypersingular{\region{}}$ is understood as its boundary trace.
The Calder\'on relations give the symmetric representation of the interior regional
Dirichlet-to-Neumann map (see \cite[Eq.~(6.43)]{S08})
\begin{equation}\label{eq:regional-dtn-bie}
  \sol{\region{}}^{-}
  =\operarorhypersingular{\region{}}+
   \operatorJ{\region{}}^*
   \operarorsingle{\region{}}^{-1}
   \operatorJ{\region{}}
   \sentence{with}
   \operatorJ{\region{}} := \tfrac12\identitymapbdd+\operarordouble{\region{}}.
\end{equation}
The continuous Maue identity is
\begin{equation}\label{eq:maue-continuous}
  \productdual{\operarorhypersingular{\region{}}\phi}{\psi}{\Gamma_{\region{}}}
  =\productdual{\operarorsingle{\region{}}\partial_s\phi}{\partial_s\psi}{\Gamma_{\region{}}}
  \sentence{for all}\phi,\psi\in H^{\frac12}(\Gamma_{\region{}}),
\end{equation}
where $\partial_s$ is the scalar tangential derivative with respect to the
arc-length of $\Gamma_{\region{}}$, taken in the direction of $\bv\tau_{\region{}}
:=\bv\nu_{\region{}}^{\perp}$ (see \cite[Theorem~6.15]{S08}).

For the unbounded region, let $\mathcal C_{\region{}}$ denote the zero-total-charge
inverse of $\operarorsingle{\region{}}$: for $g\in H^{1/2}(\Gamma_{\region{}})$,
$\rho=\mathcal C_{\region{}}g$ is the first component of the following saddle system,
where $q_{\region{}}\rho:=\left\langle\rho,\mathbf1\right\rangle_{\Gamma_{\region{}}}$:
\begin{equation}\label{eq:exterior-constrained-continuous}
  \begin{pmatrix}
    \operarorsingle{\region{}}& \mathbf1\\
    q_{\region{}} & 0
  \end{pmatrix}
  \begin{pmatrix}\rho\\c\end{pmatrix}
  =\begin{pmatrix}g\\0\end{pmatrix}.
\end{equation}
Thus, $q_{\region{}}\rho=0$.
The exterior map returning the outward flux $\partial_{\bv\nu_{\region{}}}u$ is
\begin{equation}\label{eq:regional-dtn-exterior-bie}
  \sol{\region{}}^{+}
  =\operarorhypersingular{\region{}}+
   \operatorJ{\region{}}^*\mathcal C_{\region{}}\operatorJ{\region{}}.
\end{equation}
These continuous regional maps assemble to $\dtn{\Gamma}$ in \eqref{eq:dtn}.

We now pass to their Galerkin realization.  Let $\curvepolygon{}$ be the polygonal network and let
$\totaln$ denote the number of its distinct vertices on $\curvepolygon{}$, so that all global vertex vector
belongs to $\R^{2\totaln}$. For each curve $\curvepolygon{i}$,
let $N_i$ be the number of all vertices on $\curvepolygon{i}$,
including both endpoints when $\Gamma_i$ is open and counting the periodic endpoint only once
when it is closed. We set
\[
  N_a:=\sum_{i=1}^{I_S}N_i,
\]
and thus junction vertices are counted only once in $\totaln$ and once on every incident curve in $N_a$;
in particular, for a network containing only triple junctions,
$N_a=\totaln+2I_T$ holds.

Enumerate the distinct vertices of $\curvepolygon{}$ as $\{\bv X_j\}_{j=1}^{\totaln}$ and the vertices of
$\curvepolygon{i}\,(1\leq i\leq I_S)$ in curve order as $\{\bv X_{i,k}\}_{k=1}^{N_i}$.  Let
\[
  \iota_i:\{1,\ldots,N_i\}\longrightarrow\{1,\ldots,\totaln\}
\]
be the injective local-to-global index map determined by
$\bv X_{i,k}=\bv X_{\iota_i(k)}$.  Different maps $\iota_i$ take the same value precisely at a
vertex shared by the corresponding three curves at triple junction points.
For each $1\leq i\leq I_S$, define the curve-restriction matrix
$P_i\in\R^{N_i\times\totaln}$ component-wise by
\begin{equation}\label{eq:curve-restriction}
  (P_i)_{kj}:=\delta_{\iota_i(k),j},
  \qquad 1\le k\le N_i,\quad 1\le j\le\totaln.
\end{equation}
Define the stacked restriction matrix
\[
  P_h:=
  \begin{pmatrix}P_1\\ \vdots\\ P_{I_S}\end{pmatrix}
  \in\R^{N_a\times\totaln}.
\]
Using the phase--curve signs $m_{p,i}$ of \eqref{eq:Mp}, set
\[
  M_{p,h}:=\operatorname{diag}
  \bigl(m_{p,1}\identitymatrix{N_1},\ldots,
        m_{p,I_S}\identitymatrix{N_{I_S}}\bigr)
  \in\R^{N_a\times N_a},
\]
and define the discrete signed phase--interface incidence matrix by
\begin{equation}\label{eq:discrete-incidence}
  \mathcal D_h
  :=
  \begin{pmatrix}
    P_h^{\top}M_{0,h}\\
    \vdots\\
    P_h^{\top}M_{I_P-1,h}
  \end{pmatrix}
  \in\R^{I_P\totaln\times N_a}.
\end{equation}
For each $0\leq p\leq I_P-1$, let $\checmicalDisc{p}$ be the vector of nodal values of $\chemical{p}$.
Then, for $\checmicalDiscVec{}=(\checmicalDisc{0};\ldots;\checmicalDisc{I_P-1})\in\R^{I_P\totaln}$, we have
\begin{equation}\label{eq:discrete-incidence-action}
  \mathcal D_h^{\top}\checmicalDiscVec{}
  =\sum_{p=0}^{I_P-1}M_{p,h}P_h\checmicalDisc{p}\in\R^{N_a},
\end{equation}
whose restriction to an interface $\curvepolygon{i}$
oriented from phase $p$ to phase $q$ corresponds to the nodal
jump $\checmicalDisc{q} - \checmicalDisc{p}$.
Moreover, we have
\begin{equation}
  \label{eq:mph-sum-zero}
\sum_{p=0}^{I_P-1} M_{p,h}= O \in\R^{N_a\times N_a}
\end{equation}
since each interface has two adjacent phases with
opposite incidence signs.
For consecutive vertices of $\curvepolygon{i}$, set
$h_{i,k}:=|\bv X_{i,k+1}-\bv X_{i,k}|$.  Define the lumped weight at its $k$-th vertex by
\begin{equation}
  \label{eq:arcwise-mass}
  \mu_{i,k}:=
  \begin{cases}
    \tfrac12(h_{i,k-1}+h_{i,k}) & \qquad\text{if}\quad\curvepolygon{i}\text{ is closed},\\[2pt]
    \tfrac12h_{i,1} & \qquad\text{if}\quad\curvepolygon{i}\text{ is open and }k=1,\\[2pt]
    \tfrac12(h_{i,k-1}+h_{i,k}) & \qquad\text{if}\quad\curvepolygon{i}\text{ is open and }1<k<N_i,\\[2pt]
    \tfrac12h_{i,N_i-1} & \qquad\text{if}\quad\curvepolygon{i}\text{ is open and }k=N_i,
  \end{cases}
\end{equation}
where the indices are periodic in the closed case, namely we use the convention that
$\bv X_{i,N_i+1} = \bv X_{i,1}$ and $\bv X_{i,0} = \bv X_{i,N_i}$.
The arc-wise lumped mass matrix is defined by
\begin{equation}\label{eq:arcwise-lumped-mass}
  \Mv:=\operatorname{diag}
  \bigl(\mu_{1,1},\ldots,\mu_{1,N_1},\ldots,
        \mu_{I_S,1},\ldots,\mu_{I_S,N_{I_S}}\bigr)
  \in\R^{N_a\times N_a}.
\end{equation}

For a region $\region{}$, let
$\curvepolygon{\region{}}$ denote the polygonal approximation of
$\partial \region{}$ with the orientation induced by
the outward normal $\bv\nu_{\region{}}$, let
$\edges{R}{}$ be the set of all edges forming $\curvepolygon{\region{}}$,
and for $k \in\{0,1\}$, let $\mathbb P_k(E)$ denote the polynomials of degree
at most $k$ on $E\in\edges{R}{}$.  Set
\begin{equation*}
  \begin{aligned}
  \spaceAffineGalerkin{\region{}} &:=\left\{v_h\in C(\curvepolygon{\region{}})\biggm|v_h\lfloor_{E}\in\mathbb P_1(E)
    \ \text{for all }E\in\edges{R}{}\right\},\\
  \spaceConstGalerkin{\region{}} &:=\left\{\psi_h\in L^2(\curvepolygon{\region{}})\biggm|\psi_h\lfloor_E\in\mathbb P_0(E)
    \ \text{for all }E\in\edges{R}{}\right\}.
  \end{aligned}
\end{equation*}
Let $N_X = N_X(\region{}):=\dim \spaceAffineGalerkin{\region{}}$
and $N_Y = N_Y(\region{}):=\dim \spaceConstGalerkin{\region{}}$, and enumerate the vertices and edges as
$\{\bv{x}_j\}_{j=1}^{N_X}$ and $\{E_i\}_{i=1}^{N_Y}$.
In the sequel, we suppress the arguments of $N_X(\cdot)$ and $N_Y(\cdot)$ unless it is possibly confusing for simplicity.
Let
$\{\varphi_j\}_{j=1}^{N_X}$ be the nodal basis of $\spaceAffineGalerkin{\region{}}$, defined by
$\varphi_j(\bv{x}_k)=\delta_{jk}$ (Kronecker's delta), and let
$\{\chi_i\}_{i=1}^{N_Y}$ be the edgewise-constant basis of $\spaceConstGalerkin{\region{}}$, defined by
$\chi_i=\mathbf{1}_{E_i}$.
We now introduce the matrices
$\matrixsingle{\region{}}\in\R^{N_Y\times N_Y}$, $\matrixdouble{\region{}}\in\R^{N_Y\times N_X}$, $\matrixhypersingular{\region{}}\in\R^{N_X\times N_X}$, and $\matrixcross{\region{}}\in\R^{N_Y\times N_X}$ as follows:
\begin{equation}\label{eq:bem-matrices}
\begin{aligned}
  (\matrixsingle{\region{}})_{ij}&:=\productdual{\chi_i}{\operarorsingle{\region{}}\chi_j}{\curvepolygon{\region{}}},&
  (\matrixdouble{\region{}})_{ij}&:=\productdual{\chi_i}{\operarordouble{\region{}}\varphi_j}{\curvepolygon{\region{}}},\\
  (\matrixhypersingular{\region{}})_{ij}&:=\productdual{\varphi_i}{\operarorhypersingular{\region{}}\varphi_j}{\curvepolygon{\region{}}},&
  (\matrixcross{\region{}})_{ij}&:=\productll{\chi_i}{\varphi_j}{\curvepolygon{\region{}}}.
\end{aligned}
\end{equation}
Let $\matrixD{\region{}}\in\R^{N_Y\times N_X}$ be the coefficient matrix of the
element-wise tangential derivative, defined by
\begin{equation}\label{eq:tangential-derivative-matrix}
  \partial_s\varphi_j=\sum_{i=1}^{N_Y}(\matrixD{\region{}})_{ij}\chi_i
  \qquad\text{for}\quad 1\leq j \leq N_X.
\end{equation}
For the bounded region, the symmetric Galerkin discretization of \eqref{eq:regional-dtn-bie} is
therefore
\begin{equation}\label{eq:S}
  \solintGalerkin{\region{}}
  :=\matrixhypersingular{\region{}}+(\matrixJ{\region{}})^{\top}(\matrixsingle{\region{}})^{-1}\matrixJ{\region{}}
  \with
  \matrixJ{\region{}}:=\frac12 \matrixcross{\region{}}+ \matrixdouble{\region{}},
\end{equation}
when $\matrixsingle{\region{}}$ is invertible,
and the Galerkin realization of \eqref{eq:maue-continuous} is
\begin{equation}\label{eq:maue-discrete}
  \matrixhypersingular{\region{}}=(\matrixD{\region{}})^{\top}\matrixsingle{\region{}}\matrixD{\region{}}.
\end{equation}

Meanwhile, for the unbounded region, set
$\matrixl{\region{}}:=\matrixcross{\region{}}\vectorunit{N_X}$ and let
$\matrixC{\region{}}y$ be the first component of the solution of the discrete counterpart
of \eqref{eq:exterior-constrained-continuous},
\begin{equation}\label{eq:exterior-constrained-inverse-saddle}
 \begin{pmatrix}
  \matrixsingle{\region{}} & \matrixl{\region{}}\\
  (\matrixl{\region{}})^{\top} & 0
 \end{pmatrix}
 \begin{pmatrix}\rho\\ c\end{pmatrix}
 =\begin{pmatrix}y\\ 0\end{pmatrix}.
\end{equation}
This system is uniquely solvable with no condition on
$\operatorname{cap}(\curvepolygon{\region{}})$, see Lemma~\ref{lem:zero-mean} below, and it is
the form we use for the exterior block.  Eliminating $c$ recovers the closed form:
\begin{equation}\label{eq:exterior-constrained-inverse}
  \matrixC{\region{}}= (\matrixsingle{\region{}})^{-1}
  -\frac{(\matrixsingle{\region{}})^{-1}\matrixl{\region{}}(\matrixl{\region{}})^{\top}(\matrixsingle{\region{}})^{-1}}
  {(\matrixl{\region{}})^{\top}(\matrixsingle{\region{}})^{-1}\matrixl{\region{}}},
\end{equation}
when $\matrixsingle{\region{}}$ is invertible. Define
\[
  \solextGalerkin{\region{}}:=\matrixhypersingular{\region{}}
  +(\matrixJ{\region{}})^{\top}\matrixC{\region{}}\matrixJ{\region{}}.
\]
Then, we have
\begin{equation}
  \label{eq:zero-total-charge-h}
\matrixC{\region{}}\matrixl{\region{}}=0.
\end{equation}

Let $\projGalerkin{\region{}}$ be the restriction of a global nodal trace on $\curvepolygon{}$ to $\spaceAffineGalerkin{\region{}}$, and set

\[
\solGalerkin{\region{}} =\begin{cases}
  \solintGalerkin{\region{}}\qquad&\text{if}\quad \region{}\ \text{is bounded,}\\
  \solextGalerkin{\region{}}\qquad&\text{otherwise}.
\end{cases}
\]
The discrete counterpart of
\eqref{eq:dtn} is now defined by
\begin{equation}\label{eq:discrete-dtn}
  \dtn{\curvepolygon{}}:=\sum_{\ell=1}^{I_R}(\projGalerkin{\region{\ell}})^{\top}\solGalerkin{\region{\ell}}\projGalerkin{\region{\ell}}.
\end{equation}

Before proving that the matrix $\dtn{\curvepolygon{}}$ is positive semi-definite,
we prepare a quantity of an interface $\Gamma$.

\begin{definition}[Logarithmic energy and capacity]
  \label{dfn:log-energy-capacity}
For a Borel measure $\mu$ on $\R^2$, the logarithmic energy of $\mu$ is defined by
\[
I(\mu) := 2\pi\int_{\R^2}\int_{\R^2}G(\bv x - \bv y)\dmu(\bv x)\dmu(\bv y).
\]
Let $\Gamma\subset\R^2$ be a compact boundary and
$\mathcal P(\Gamma)$ the set of all probability measures $\mu$ on $\R^2$ such that $\operatorname{supp}\mu\subseteq\Gamma$.
Then, the logarithmic capacity of $\Gamma$ is defined by
\[
  \operatorname{cap}(\Gamma)
  :=\exp\!\left(
  -\inf_{\mu\in\mathcal P(\Gamma)}
  I(\mu)
  \right).
\]
\end{definition}

With the aid of the logarithmic capacity introduced above,
we recall a sufficient condition to guarantee the positive definiteness of the Galerkin matrix of the single-layer operator.
\begin{proposition}
  \label{prop:capacity}
Assume that $\operatorname{cap}(\curvepolygon{\region{}})<1$. Then,
the single-layer operator $\operarorsingle{\region{}}$ on $H^{-1/2}(\curvepolygon{\region{}})$ is coercive.
Consequently,
its Galerkin matrix $\matrixsingle{\region{}}$ is symmetric and positively definite.
\end{proposition}
\begin{proof}
  The coercivity of the single-layer operator follows from \cite[pp.~142--143, Theorem~6.23]{S08}.
  Thus, there exists a constant $c_V>0$ such that
  \[
    \productdual{\psi}{\operarorsingle{\region{}}\psi}{\curvepolygon{\region{}}}
    \geq c_V\|\psi\|_{H^{-1/2}(\curvepolygon{\region{}})}^2
    \qquad\text{for every }\psi\in H^{-\frac{1}{2}}(\curvepolygon{\region{}}).
  \]
  For $y=(y_i)_{i=1}^{N_Y}\in\R^{N_Y}$, set
  \[
    \psi_h:=\sum_{i=1}^{N_Y}y_i\chi_i
    \in \spaceConstGalerkin{\region{}}\subset H^{-1/2}(\curvepolygon{\region{}}).
  \]
  By the definition of the Galerkin matrix in \eqref{eq:bem-matrices},
  \[
    y^\top\matrixsingle{\region{}}y
    =\sum_{i,j=1}^{N_Y}y_i y_j
      \productdual{\chi_i}{\operarorsingle{\region{}}\chi_j}{\curvepolygon{\region{}}}
    =\productdual{\psi_h}{\operarorsingle{\region{}}\psi_h}{\curvepolygon{\region{}}}
    \geq c_V\|\psi_h\|_{H^{-1/2}(\curvepolygon{\region{}})}^2.
  \]
  Since $\{\chi_i\}_{i=1}^{N_Y}$ is a basis of $\spaceConstGalerkin{\region{}}$,
  $y \neq 0$ implies $\psi_h\neq 0$, and hence the right-hand side of the above inequality is positive.
  Therefore, we conclude that $\matrixsingle{\region{}}$ is positive definite.
  Moreover, the symmetry of $\matrixsingle{\region{}}$ follows from $G(\bv x-\bv y)=G(\bv y-\bv x)$.
  The proof is now complete.
\end{proof}

We now recall a criterion to satisfy the hypothesis of Proposition~\ref{prop:capacity}, which can be confirmed much easier.
\begin{lemma}
  \label{lem:diam-cap}
If $\operatorname{diam}(\Gamma)<2$, then $\operatorname{cap}(\Gamma) < 1$ holds.
\end{lemma}
\begin{proof}
  See \cite[Eq.(1.2)]{BPS02}.

\end{proof}

We note that the capacity condition of Proposition~\ref{prop:capacity} restricts the size of $\curvepolygon{\region{}}$.
The exterior block evaluates the quadratic form of $\matrixsingle{\region{}}$ only at vectors $y$ with
$(\matrixl{\region{}})^{\top}y=0$, and at those points, it is positive regardless of the logarithmic capacity.
\begin{lemma}\label{lem:zero-mean}
Let $\region{}$ be a polygonal region.  Then, it follows that
\begin{equation}\label{eq:ell-identities}
 (\matrixl{\region{}})^{\top}y
 =\int_{\curvepolygon{\region{}}}\psi_h\dH{1}
 \sentence{for}\psi_h=\sum_{i=1}^{N_Y}y_i\chi_i,
 \qquad
 (\matrixl{\region{}})^{\top}\matrixD{\region{}}=\vectorzero{N_X}^{\top}.
\end{equation}
Moreover, it holds that
\begin{equation}\label{eq:zero-mean-matrix}
 y^{\top}\matrixsingle{\region{}}y>0
 \qquad\text{for every}\quad y\in\R^{N_Y}\setminus\{\vectorzero{N_Y}\}
 \quad\text{with}\quad (\matrixl{\region{}})^{\top}y=0,
\end{equation}
and consequently the matrix of \eqref{eq:exterior-constrained-inverse-saddle} is invertible.
\end{lemma}
\begin{proof}
  The first equation of \eqref{eq:ell-identities} is straightforward from the definition of $\matrixl{\region{}}$.
Since $\{\varphi_j\}_{j=1}^{N_X}$ is a partition of unity on $\curvepolygon{\region{}}$, the definition
of $\matrixcross{\region{}}$ in \eqref{eq:bem-matrices} gives
\begin{equation}
  \label{eq:zero-mean-1}
(\matrixl{\region{}})_i=\int_{E_i}\sum_{j=1}^{N_X}\varphi_j\dH{1}=|E_i|.
\end{equation}
Let $\phi_h\in\spaceAffineGalerkin{\region{}}$ have the coefficient vector $x\in\R^{N_X}$
and let $\bv a_i,\,\bv b_i$ be the endpoints of $E_i$ in the direction of the orientation.
Then, since $\phi_h$ is affine, \eqref{eq:tangential-derivative-matrix} gives
$$|E_i|(\matrixD{\region{}}x)_i=\phi_h(\bv b_i)-\phi_h(\bv a_i),$$
and summing over the edges telescopes on each closed component of $\curvepolygon{\region{}}$ together with \eqref{eq:zero-mean-1} gives
$(\matrixl{\region{}})^{\top}\matrixD{\region{}}x=0$ for every $x\in\R^{N_X}$, and hence the second equation in \eqref{eq:ell-identities} follows.

Suppose that $\psi\in H^{-\frac12}(\curvepolygon{\region{}})$ satisfies
$\productdual{\psi}{1}{\curvepolygon{\region{}}}=0$ and let
$u$ be the single-layer potential of $\psi$, that is,
\begin{equation}\label{eq:zero-mean-potential}
 u(\bv x):=\int_{\curvepolygon{\region{}}}G(\bv x-\bv y)\psi(\bv y)\dH{1}(\bv y),
 \qquad\bv x\in\R^2\setminus\curvepolygon{\region{}}.
\end{equation}
The above $u$ can be extended to the whole space $\R^2$ in the trace sense;
observe that $u$ satisfies the jump relation $\jump{\partial_{\normal_{\region{}}}u}=-\psi$,
where this jump is defined by \eqref{eq:jump} replaced $\normal_i$ with $\normal_{\region{}}$.

Then, we compute
\begin{equation}\label{eq:zero-mean-farfield}
 u(\bv x)= u(\bv x) - \left(-\frac{1}{2\pi}\log|\bv x|\productdual{\psi}{1}{\curvepolygon{\region{}}}\right)
 =-\frac1{2\pi}\int_{\curvepolygon{\region{}}}
 \log\frac{|\bv x-\bv y|}{|\bv x|}\,\psi(\bv y)\dH{1}(\bv y).
\end{equation}
The logarithm in \eqref{eq:zero-mean-farfield} is $O(|\bv x|^{-1})$ uniformly in
$\bv y\in\curvepolygon{\region{}}$, and hence
\begin{equation}\label{eq:zero-mean-decay}
 u(\bv x)=O\left(\frac1{|\bv x|}\right),
 \qquad
 \nabla u(\bv x)=O\left(\frac1{|\bv x|^2}\right)
 \qquad\text{as}\quad|\bv x|\to\infty .
\end{equation}
We now take $L>0$ be so large that $\curvepolygon{\region{}}\subset B_L(\bv 0)$.
Since $u$ is harmonic in $\R^2\setminus\curvepolygon{\region{}}$ and continuous across it,
Green's identity on each
connected component of $B_L(\bv 0)\setminus\curvepolygon{\region{}}$ gives, after summing
the two traces on $\curvepolygon{\region{}}$,
\begin{equation*}
  \begin{aligned}
 \int_{B_L(\bv 0)\setminus\curvepolygon{\region{}}}|\nabla u|^2\dL{2}
 &=-\int_{\curvepolygon{\region{}}}
   \jump{\partial_{\normal_{\region{}}}u}\,u\dH{1}
 +\int_{\partial B_L(\bv 0)}u\,\frac{\partial u}{\partial\normal_{B_L(\bv 0)}}\dH{1}\\
  &= \productdual{\psi}{\operarorsingle{\region{}}\psi}{\curvepolygon{\region{}}} + O\left(\frac1{L^2}\right)
  \sentence{as}L\to\infty.
 \end{aligned}
\end{equation*}
Here, we have invoked that the trace of $u$ equals $\operarorsingle{\region{}}\psi$ and
the bounds given in \eqref{eq:zero-mean-decay}.
Letting $L\to\infty$, we obtain
\begin{equation}
  \label{eq:zero-mean-2}
 \productdual{\psi}{\operarorsingle{\region{}}\psi}{\curvepolygon{\region{}}}
 =\int_{\R^2}|\nabla u|^2\dL{2}\geq0 .
\end{equation}
If the right-hand side of \eqref{eq:zero-mean-2} vanishes, then $\nabla u=\bv 0$ almost
everywhere in $\R^2$.  The continuity of $u$ across $\curvepolygon{\region{}}$ places $u$
in $H^1_{\mathrm{loc}}(\R^2)$, and $\R^2$ is connected, so $u$ must be constant;
the decay condition \eqref{eq:zero-mean-decay} then forces $u\equiv0$.
The jump relation consequently gives $\psi=-\jump{\partial_{\normal_{\region{}}}u}=0$.

Suppose that $y\in\R^{N_Y}\setminus\{\vectorzero{N_Y}\}$ satisfies $(\matrixl{\region{}})^{\top}y=0$.
Then, $\psi_h=\sum_iy_i\chi_i\in\spaceConstGalerkin{\region{}}$ is nonzero and,
by the first identity in \eqref{eq:ell-identities}, has vanishing total charge;
that is,
\[
 \productdual{\psi_h}{1}{\curvepolygon{\region{}}}
 =\int_{\curvepolygon{\region{}}}\psi_h\dH{1}
 =(\matrixl{\region{}})^{\top}y=0 .
\]
Meanwhile, the definition of $\matrixsingle{\region{}}$ in \eqref{eq:bem-matrices} gives
\[
 y^{\top}\matrixsingle{\region{}}y
 =\sum_{i=1}^{N_Y}\sum_{j=1}^{N_Y}y_iy_j
  \productdual{\chi_i}{\operarorsingle{\region{}}\chi_j}{\curvepolygon{\region{}}}
 =\productdual{\psi_h}{\operarorsingle{\region{}}\psi_h}{\curvepolygon{\region{}}}.
\]
By the previous argument, noting that $\productdual{\psi_h}{1}{\curvepolygon{\region{}}} = 0$,
the right-hand side of the above inequality is positive thanks to the identity of \eqref{eq:zero-mean-2};
otherwise it must hold that $\psi_h = 0$, and this is a contradiction.
Therefore, \eqref{eq:zero-mean-matrix} follows.

It remains to prove that the matrix of \eqref{eq:exterior-constrained-inverse-saddle} is
invertible.  Let $(\rho,c)\in\R^{N_Y}\times\R$ be annihilated by it, that is,
\[
 \matrixsingle{\region{}}\rho+c\,\matrixl{\region{}}=\vectorzero{N_Y},
 \qquad
 (\matrixl{\region{}})^{\top}\rho=0.
\]
Multiplying the first equation by $\rho^{\top}$ from the left and inserting the second one,
we obtain
\[
 \rho^{\top}\matrixsingle{\region{}}\rho
 =-c\,(\matrixl{\region{}})^{\top}\rho=0.
\]
Since $\rho$ satisfies $(\matrixl{\region{}})^{\top}\rho=0$, the inequality
\eqref{eq:zero-mean-matrix} forces $\rho=\vectorzero{N_Y}$.
The first equation then reads
$c\,\matrixl{\region{}}=\vectorzero{N_Y}$, and $\matrixl{\region{}}\neq\vectorzero{N_Y}$ by
\eqref{eq:zero-mean-1}, whence $c=0$.  The proof is now complete.
\end{proof}

\begin{lemma}\label{lem:discrete-regional-dtn-structure}
Let $\region{}$ be a connected polygonal region whose boundary is a finite disjoint
union of simple closed polygons, and suppose, in the case that $\region{}$ is
bounded, that $\matrixsingle{\region{}}$ is symmetric
positive definite.
Then, for either the interior or the exterior block defined in
\eqref{eq:S} and \eqref{eq:exterior-constrained-inverse-saddle}, it follows that
\begin{subequations}\label{eq:discrete-sol-structure}
  \begin{align}
    (\solGalerkin{\region{}})^{\top} &= \solGalerkin{\region{}}\label{eq:discrete-sol-structure-a},\\
    x^\top \solGalerkin{\region{}} x &\ge 0\label{eq:discrete-sol-structure-b},\\
  \ker \solGalerkin{\region{}} &= \operatorname{span}\{\vectorunit{N_X}\}\label{eq:discrete-sol-structure-c}.
  \end{align}
\end{subequations}
\end{lemma}
\begin{proof}
We first prove \eqref{eq:discrete-sol-structure-a} and
\eqref{eq:discrete-sol-structure-b}.
The Maue formula \eqref{eq:maue-discrete} gives
\[
 \bigl(\matrixhypersingular{\region{}}\bigr)^\top
 =\matrixhypersingular{\region{}},
 \qquad
 x^\top\matrixhypersingular{\region{}}x
 =(\matrixD{\region{}}x)^\top\matrixsingle{\region{}}(\matrixD{\region{}}x)\geq0,
\]
where the inequality follows from the assumption on
$\matrixsingle{\region{}}$ when $\region{}$ is bounded and, in general, from
$(\matrixl{\region{}})^{\top}\matrixD{\region{}}=\vectorzero{N_X}^{\top}$ in
\eqref{eq:ell-identities} together with \eqref{eq:zero-mean-matrix}.
Consequently, the interior formula \eqref{eq:S} yields
\begin{equation}\label{eq:interior-block-quadratic-form}
\begin{aligned}
 \bigl(\solintGalerkin{\region{}}\bigr)^\top
 &=\solintGalerkin{\region{}},\\
 x^\top\solintGalerkin{\region{}}x
 &=(\matrixD{\region{}}x)^\top\matrixsingle{\region{}}(\matrixD{\region{}}x)
 +(\matrixJ{\region{}}x)^\top(\matrixsingle{\region{}})^{-1}(\matrixJ{\region{}}x)
 \geq0.
\end{aligned}
\end{equation}

For the exterior block, let $y_1,\,y_2\in\R^{N_Y}$ and let $(\rho_1,c_1)\in\R^{N_Y}\times\R$ and
$(\rho_2,c_2)\in\R^{N_Y}\times\R$ be the corresponding solutions of
\eqref{eq:exterior-constrained-inverse-saddle},
so that $\matrixC{\region{}}y_i=\rho_i$ for $i=1,2$.
Using the two equations of
\eqref{eq:exterior-constrained-inverse-saddle} in turn, we obtain
\[
 y_1^{\top}\matrixC{\region{}}y_2
 =\bigl(\matrixsingle{\region{}}\rho_1+c_1\matrixl{\region{}}\bigr)^{\top}\rho_2
 =\rho_1^{\top}\matrixsingle{\region{}}\rho_2,
\]
which is symmetric in $y_1$ and $y_2$;
hence $\matrixC{\region{}}$ is symmetric.
Taking $y_1=y_2=y$ and invoking \eqref{eq:zero-mean-matrix}, which applies because
$(\matrixl{\region{}})^{\top}\rho=0$, gives
\begin{equation}\label{eq:constrained-inverse-quadratic-form}
 y^{\top}\matrixC{\region{}}y=\rho^{\top}\matrixsingle{\region{}}\rho\geq0,
\end{equation}
with equality precisely when $\rho=\vectorzero{N_Y}$, that is, when
$y=c\,\matrixl{\region{}}$ (see again \eqref{eq:exterior-constrained-inverse-saddle}).
Therefore, we obtain
\begin{equation}\label{eq:constrained-inverse-kernel}
 \ker \matrixC{\region{}} = \operatorname{span}\{\matrixl{\region{}}\},
\end{equation}
which contains \eqref{eq:zero-total-charge-h} as the case $(\rho,c)=(\vectorzero{N_Y},1)$.
It now follows from \eqref{eq:maue-discrete}, \eqref{eq:exterior-constrained-inverse-saddle}, and
\eqref{eq:constrained-inverse-quadratic-form} that
\begin{equation}\label{eq:exterior-block-quadratic-form}
\begin{aligned}
 \bigl(\solextGalerkin{\region{}}\bigr)^\top
 &=\solextGalerkin{\region{}},\\
 x^\top\solextGalerkin{\region{}}x
 &=(\matrixD{\region{}}x)^\top\matrixsingle{\region{}}(\matrixD{\region{}}x)
 +(\matrixJ{\region{}}x)^\top \matrixC{\region{}}
   (\matrixJ{\region{}}x)
 \geq0.
\end{aligned}
\end{equation}
Therefore, \eqref{eq:discrete-sol-structure-a} and \eqref{eq:discrete-sol-structure-b} follow from \eqref{eq:interior-block-quadratic-form} and \eqref{eq:exterior-block-quadratic-form}.

It remains to prove \eqref{eq:discrete-sol-structure-c}.
Let us show that $\operatorname{span}\{\vectorunit{N_X}\}\subseteq\ker\solintextGalerkin{\region{}}$.
Since the nodal basis is a partition of unity, we have
\[
 \sum_{j=1}^{N_X}\varphi_j\equiv1\sentence{on}\curvepolygon{\region{}},
 \qquad
 \matrixD{\region{}}\mathbf1=0.
\]
With the outward normal of $\region{}$, the double-layer trace relations
\cite[Lemma~6.11, Eq.~(6.14), and the subsequent exterior-trace formula]{S08}
give
\[
 \operarordouble{\region{}}\mathbf1=
 \begin{cases}
  -\frac12\mathbf1&\qquad\region{}\text{ is bounded},\\
  +\frac12\mathbf1&\qquad\region{}\text{ is unbounded}.
 \end{cases}
\]
Therefore, for every $E_i$, we have
\[
 \begin{aligned}
 (\matrixdouble{\region{}}\vectorunit{N_X})_i
 &=\sum_{j=1}^{N_X}
   \productdual{\chi_i}{\operarordouble{\region{}}\varphi_j}{\curvepolygon{\region{}}}
 =\productdual{\chi_i}{\operarordouble{\region{}}\mathbf1}{\curvepolygon{\region{}}}.
 \end{aligned}
\]
Consequently, it follows that
\begin{equation}\label{eq:discrete-calderon-constant}
 \matrixdouble{\region{}}\vectorunit{N_X}=
 \begin{cases}
  -\frac12\matrixcross{\region{}}\vectorunit{N_X},&\qquad\region{}\text{ is bounded},\\
  +\frac12\matrixcross{\region{}}\vectorunit{N_X},&\qquad\region{}\text{ is unbounded},
 \end{cases}
 \qquad
 \matrixD{\region{}}\vectorunit{N_X}=0.
\end{equation}
If $\region{}$ is bounded, then
\[
 \matrixJ{\region{}}\vectorunit{N_X}=0,
 \qquad
 \matrixhypersingular{\region{}}\vectorunit{N_X}
 =(\matrixD{\region{}})^{\top}\matrixsingle{\region{}}\matrixD{\region{}}\vectorunit{N_X}=0,
\]
and hence \eqref{eq:S} gives
\[
 \solintGalerkin{\region{}}\vectorunit{N_X}=0.
\]
If $\region{}$ is unbounded,
then \eqref{eq:exterior-constrained-inverse} and \eqref{eq:zero-total-charge-h} give
\[
 \matrixJ{\region{}}\vectorunit{N_X}
 =\matrixcross{\region{}}\vectorunit{N_X}=\matrixl{\region{}},
 \qquad
 \matrixC{\region{}}\matrixl{\region{}}=0,
 \qquad
 \matrixhypersingular{\region{}}\vectorunit{N_X}=0,
\]
and therefore
\[
 \solextGalerkin{\region{}}\vectorunit{N_X}=0.
\]

Conversely, we shall prove that $\operatorname{span}\{\mathbf 1\} \supseteq\ker\solintextGalerkin{\region{}}$.
Suppose that $x\in\ker\solintextGalerkin{\region{}}$, and let
$\phi_h\in \spaceAffineGalerkin{\region{}}$ be the trace with the coefficient vector $x\in \R^{N_X}$.

For the interior block, we may suppose that $\region{}$ is bounded.
Then, we can write

\[
\partial\region{} = \boundarycomponent{0}\cup\bigcup_{i=1}^m \boundarycomponent{i},
\]
where $\boundarycomponent{0}$ borders the unbounded component of $\R^2\setminus\region{}$, while the remainder $\boundarycomponent{i}$'s border
the $m$ bounded components of $\R^2\setminus\region{}$; here, we use the convention that $\partial\region{} = \boundarycomponent{0}$ in the case $m = 0$.
Since $x^\top\solintGalerkin{\region{}}x=0$,
\eqref{eq:interior-block-quadratic-form} together with the positive definiteness of $\matrixsingle{\region{}}$ and $(\matrixsingle{\region{}})^{-1}$ gives
\[
 \matrixD{\region{}}x=0,
 \qquad
 \matrixJ{\region{}}x=0.
\]
Since $\phi_h$ has zero tangential derivative,
it follows that $\phi_h\lfloor_{\boundarycomponent{j}}\equiv c_j$ for $0\leq j\leq m$.
The component-wise constants give the decomposition:
\begin{equation}
  \label{eq:regional-dtn-1}
 \phi_h=c_0\functionCharacteristic{\boundarycomponent{0}}
 +\sum_{j=1}^m c_j\functionCharacteristic{\boundarycomponent{j}}
 \qquad\text{on}\quad \partial \region{}.
\end{equation}
We note that any point in $\region{}$ is surrounded by the outer boundary $\boundarycomponent{0}$,
whereas any bounded component with hole boundary $\boundarycomponent{j}$ does not include it.

We now recall from
\eqref{eq:bie-operators} that
$\operarordoubleregional{\region{}}$ denotes the off-boundary double-layer potential
to distinguish it from its boundary operator $\operarordouble{\region{}}$:
\[
 (\operarordoubleregional{\region{}}\phi)(\bv z)
 :=\int_{\partial R}
 \frac{\partial G}{\partial\bv\nu_{\region{}}(\bv y)}(\bv z-\bv y)
 \phi(\bv y)\,\dH{1}(\bv y),
 \qquad \bv z\in\R^2\setminus\partial \region{}.
\]
In particular, it follows that
\[
 \operarordoubleregional{\region{}}\functionCharacteristic{\boundarycomponent{0}}=-1,
 \qquad
 \operarordoubleregional{\region{}}\functionCharacteristic{\boundarycomponent{j}}=0
 \quad(1\leq j\leq m)
 \qquad\text{in }\region{}.
\]
By linearity, we deduce from \eqref{eq:regional-dtn-1} that
\[
 \operarordoubleregional{\region{}}\phi_h=-c_0
 \sentence{in}\region{}.
\]
Taking the interior Dirichlet trace of this identity
and using the double-layer jump formula, we obtain 
\[
 (-\tfrac12\identitymapbdd+\operarordouble{\region{}})\phi_h=-c_0\functionCharacteristic{}
 \sentence{on}\partial\region{}.
\]
Adding $\phi_h$ to both sides and restricting the result to $\boundarycomponent{j}$, we obtain
\[
 \left(\tfrac12\identitymapbdd+\operarordouble{\region{}}\right)\phi_h\lfloor_{\boundarycomponent{j}}
 =\phi_h\lfloor_{\boundarycomponent{j}}-c_0
 =c_j-c_0.
\]
By the definitions of $\matrixcross{\region{}}$, $\matrixdouble{\region{}}$, and $\matrixJ{\region{}}$ in
\eqref{eq:bem-matrices} and \eqref{eq:S}, an edge $E_i\subset\boundarycomponent{j}$
therefore satisfies
\[
 (\matrixJ{\region{}}x)_i
 =\int_{E_i}(\tfrac12\identitymapbdd+\operarordouble{\region{}})\phi_h\,\dH{1}
 =\int_{E_i}(c_j-c_0)\,\dH{1}
 =|E_i|(c_j-c_0).
\]
Since $\matrixJ{\region{}}x=0$ and $|E_i|>0$, it follows that
$c_j=c_0\ (0\leq j\leq m)$.
Consequently, we see that
\[
 \phi_h=c_0\functionCharacteristic{\partial\region{}}
 \sentence{and}
 x=c_0\vectorunit{N_X}\in\operatorname{span}\{\vectorunit{N_X}\}.
\]

For the exterior block, the assumption $x\in\ker\solextGalerkin{\region{}}$ implies
$x^\top \solextGalerkin{\region{}}x=0$.
Letting $y = \matrixD{\region{}}x$, we see that $(\matrixl{\region{}})^\top y = 0$ by Lemma~\ref{lem:zero-mean},
and \eqref{eq:zero-mean-matrix} implies that $y = 0$.
Since $\matrixC{\region{}}$ is positively semidefinite, the formula \eqref{eq:exterior-block-quadratic-form} yields that
$\matrixJ{\region{}}x\in \ker\matrixC{\region{}}$, and hence, by \eqref{eq:constrained-inverse-kernel}, we have
\[
 \matrixJ{\region{}}x=\alpha\matrixl{\region{}}
 \sentence{for some}\alpha\in\R.
\]
In this case, there is no boundary $\boundarycomponent{0}$, and hence we can write $\partial\region{}=\bigcup_{j=1}^m\boundarycomponent{j}$,
where $\boundarycomponent{j}$ bounds the $j$-th bounded component of the complement.
The outward normal $\normal_{\region{}}$ of the unbounded region points into each component.
Let $\phi_h\lfloor_{\boundarycomponent{j}}=c_j$.
We observe that the double-layer potential of $\functionCharacteristic{\boundarycomponent{j}}$ is
equal to $+1$ inside the components and zero outside it.
Therefore, its boundary traces give
\[
 \operarordouble{\region{}}\functionCharacteristic{\boundarycomponent{j}}
 =\frac12\functionCharacteristic{\boundarycomponent{j}}
 \sentence{a.e. on}\partial\region{},
\]
and thus, we have
\[
 (\tfrac12\identitymapbdd+\operarordouble{\region{}})\phi_h=\phi_h.
\]
For each edge $E_i\subset\boundarycomponent{j}$, it follows that
\[
 |E_i|c_j=(\matrixJ{\region{}}x)_i
 =\alpha(\matrixl{\region{}})_i=\alpha|E_i|,
\]
therefore $c_j=\alpha$ for every $1\leq j\leq m$, and $x=\alpha\vectorunit{N_X}$.
The proof is now complete.
\end{proof}

We now show properties of the linear operator $\dtn{\curvepolygon{}}$ which are corresponding to Proposition~\ref{prop:dtn-structure} as follows.
\begin{proposition}\label{prop:discrete-dtn-structure}
Suppose that
$\operatorname{cap}(\curvepolygon{\region{r}})<1$ for every bounded region
$\region{r}$, $1\leq r\leq I_R$.
It holds that
\begin{subequations}\label{eq:discrete-dtn-structure}
  \begin{align}
    (\dtn{\curvepolygon{}})^{\top} &= \dtn{\curvepolygon{}}\label{eq:discrete-dtn-structure-a},\\
    z^{\top}\dtn{\curvepolygon{}} z &= \sum_{r=1}^{I_R}(\projGalerkin{\region{r}}z)^{\top}\solGalerkin{\region{r}}(\projGalerkin{\region{r}}z)\geq0
    \qquad\forall z\in\R^{\totaln}\label{eq:discrete-dtn-structure-b},\\
    \ker\dtn{\curvepolygon{}} &= \operatorname{span}\{\vectorunit{\totaln}\}\label{eq:discrete-dtn-structure-c}.
  \end{align}
\end{subequations}
\end{proposition}
\begin{proof}
For every bounded region $\region{\ell}$, the assumption on the capacity of
$\curvepolygon{\region{\ell}}$ and Proposition~\ref{prop:capacity} imply that
the matrix $\matrixsingle{\region{\ell}}$ is symmetric and positively definite.
Hence, the hypothesis of Lemma~\ref{lem:discrete-regional-dtn-structure} is fulfilled for every
region, the unbounded one included.

Consequently, \eqref{eq:discrete-sol-structure-a} and the assembly formula \eqref{eq:discrete-dtn} give
\[
 \bigl(\dtn{\curvepolygon{}}\bigr)^\top
 =\sum_{r=1}^{I_R}(\projGalerkin{\region{r}})^\top
 \bigl(\solGalerkin{\region{r}}\bigr)^\top \projGalerkin{\region{r}}
 =\sum_{r=1}^{I_R}(\projGalerkin{\region{r}})^\top
 \solGalerkin{\region{r}} \projGalerkin{\region{r}}
 =\dtn{\curvepolygon{}}.
\]
Hence, \eqref{eq:discrete-dtn-structure-a} follows.
Likewise, \eqref{eq:discrete-sol-structure-b} gives, for every global nodal vector $z\in\R^{\totaln}$,
\[
 z^\top\dtn{\curvepolygon{}}z
 =\sum_{r=1}^{I_R}(\projGalerkin{\region{r}}z)^\top
 \solGalerkin{\region{r}}(\projGalerkin{\region{r}}z)
 \geq0.
\]
This proves \eqref{eq:discrete-dtn-structure-b}.

Since $\projGalerkin{\region{r}}\vectorunit{\totaln}=\vectorunit{N_X(\region{r})}$, the regional kernel identity
\eqref{eq:discrete-sol-structure-c} together with \eqref{eq:discrete-dtn} gives
\[
 \operatorname{span}\{\vectorunit{\totaln}\}
 \subseteq\ker\dtn{\curvepolygon{}}.
\]
Conversely, assume that $z\in\ker\dtn{\curvepolygon{}}$.
Then, the formula \eqref{eq:discrete-dtn-structure-b} together with \eqref{eq:discrete-sol-structure-b} implies that
\[
 (\projGalerkin{\region{r}}z)^\top
 \solGalerkin{\region{r}}(\projGalerkin{\region{r}}z)=0
 \qquad 1\leq \forall r\leq I_R.
\]
By \eqref{eq:discrete-sol-structure-a} and \eqref{eq:discrete-sol-structure-b},
we see that $\projGalerkin{\region{r}}z \in\ker\solGalerkin{\region{r}}$, and
\eqref{eq:discrete-sol-structure-c} implies that
there are constants $c_r\in\R$ such that
\[
 \projGalerkin{\region{r}}z=c_r\vectorunit{N_X(\region{r})}
 \qquad 1\leq \forall r\leq I_R.
\]
Since the nodal trace should correspond
on two regions sharing a curve as a part of their boundaries,
we see that $c_1=\cdots=c_{I_R}=:c$, and therefore,
\[
 z=c\vectorunit{\totaln},
 \qquad
 \ker\dtn{\curvepolygon{}}=\operatorname{span}\{\vectorunit{\totaln}\}.
\]
The proof is now complete.
\end{proof}

\section{Fully discrete scheme}\label{sec:scheme}
We discretize the time interval $[0,T]$ into $M$ sub-intervals.
Given a curve network $\Gamma_0\subset\R^2$, we approximate it by a polygonal curve network $\curvepolygonsec{0}$ as in Section~\ref{sec:dio};
we hereafter construct a sequence of polygonal curve networks $\{\curvepolygonsec{m}\}_{m=0}^M$ approximating a solution to the system \eqref{eq:strong},
where $\curvepolygonsec{m}$ denotes the approximate interface of $\Gamma(t)$ at $t = mT/M$ for $m=0,\ldots,M$.

For each curve $\curvepolygonsec{m}_i$, let
\[
 \femspaceS{m}{i}
 :=\left\{\phi_i\in C(\curvepolygonsec{m}_i)\biggm|
 \phi_i\lfloor_e\in\mathbb P_1(e)\quad\forall e\in\edgesDisc{m}{i}\right\},
\]
where $\edgesDisc{m}{i}$ denotes the set of all edges on the curve $\curvepolygonsec{m}_i$.
For an edge $j=[\bv X_{i,k}^m,\bv X_{i,k+1}^m]\in\edgesDisc{m}{i}$ and $\bv Z\in\R^{2\totaln}$,
define
\begin{equation}\label{eq:edge-difference}
  \begin{aligned}
 \bv\ell_j^m &:=\bv X_{i,k+1}^m-\bv X_{i,k}^m,\\
 \bv\ell_j(\bv Z)
 &:=\bv Z_{i,k+1}-\bv Z_{i,k}
 \sentence{with}
 \bv Z_{i,k}:=\bv Z_{\iota_i(k)}\in\R^2,
  \end{aligned}
\end{equation}
where the vertex index is periodic for a closed curve.
Based on this space, we define three boundary element spaces corresponding to $\femspaceSS{}{\Gamma}$, $\femspaceKK{}{\Gamma}$, and $\femspaceXX{}{\Gamma}$ in \eqref{eq:weak-bes-1} and \eqref{eq:weak-bes-2}:
\begin{align*}
 \mathbb S^m
 &: =\left\{(\phi_1,\ldots,\phi_{I_S})\in\prod_{i=1}^{I_S}\femspaceS{m}{i}\biggm|
 \phi_{\indextj{k}{1}}(\junction{k})
 =\phi_{\indextj{k}{2}}(\junction{k})
 =\phi_{\indextj{k}{3}}(\junction{k}),\ 1\leq \forall k\leq I_T\right\},\\
 \mathbb K^m
 &: =\prod_{i=1}^{I_S}\femspaceS{m}{i},\\
 \mathbb X^m
 &: =\left\{(\bv\eta_1,\ldots,\bv\eta_{I_S})
 \in\prod_{i=1}^{I_S}[\femspaceS{m}{i}]^2\biggm|
 \bv\eta_{\indextj{k}{1}}(\junction{k})
 =\bv\eta_{\indextj{k}{2}}(\junction{k})
 =\bv\eta_{\indextj{k}{3}}(\junction{k}),\ 1\leq\forall k\leq I_T\right\}.
\end{align*}
It is easily seen that $\dim \femspaceSS{m}{} = \totaln$, $\dim \femspaceKK{m}{}= N_a$, and $\dim \femspaceXX{m}{} = 2\totaln$.
For $u=(u_i)_{i=1}^{I_S},\,z=(z_i)_{i=1}^{I_S}\in\femspaceKK{m}{}$,
define the mass-lumped inner product by
\begin{equation}\label{eq:mass-lumped-inner-product}
  \productmass{u}{z}{M_v}
 :=\sum_{i=1}^{I_S}\sum_{k=1}^{N_i}
 \mu_{i,k}\,u_i(\bv X_{i,k}^m)z_i(\bv X_{i,k}^m) \approx \sum_{i = 1}^{I_S} \int_{\curvepolygonsec{m}_i} u(\bv x)z(\bv x)\,\dH{1}(\bv x),
\end{equation}
where we recall the definition of the arc-wise lumped mass $\mu_{i,k}$ and matrix $M_v$ from \eqref{eq:arcwise-mass} and \eqref{eq:arcwise-lumped-mass}, respectively.
For a vector $g\in \R^{I_P\totaln}$ satisfying $g^\top \mathbf 1_{I_P\totaln} \neq 0$, define
\begin{equation}
  \label{eq:gauge}
 \femspaceWW{m}{g}
 :=\left\{\bvp W\in(\femspaceSS{m}{})^{I_P}\biggm|g^\top\bvp W=0\right\}.
\end{equation}
In the implementation, the vector $g$ is set to the mass-lumped line-integral vector.
More precisely,
\[
 g:=\mathbf1_{I_P}\otimes\widehat m_v\qquad\text{with}\quad
 \widehat m_v:=P_h^\top M_v\mathbf1_{N_a}\in\R^{N_{\mathrm{tot}}}.
\]

We now encode the constraint of $\femspaceWW{m}{g}$ as
\begin{equation*}
 g^\top\bvp W
 =\sum_{p=0}^{I_P-1}\widehat m_v^\top W_p
 =\sum_{p=0}^{I_P-1}\sum_{j=1}^{N_{\mathrm{tot}}}
   \widehat m_{v,j}W_{p,j}=0.
\end{equation*}
For an edge of $\curvepolygonsec{m}_i$ with consecutive vertices
$\bv X_{i,k}^m$ and $\bv X_{i,k+1}^m$, we let
\begin{equation}\label{eq:edge-normal}
 \normaldiscsec{i,k}{m}
 :=-\frac{\left(\bv X_{i,k+1}^m-\bv X_{i,k}^m\right)^{\perp}}{h_{i,k}}
\end{equation}
be its unit normal, oriented as in Section~\ref{sec:setting}, that is,
pointing from $\indexphasefrom{i}$ to $\indexphaseto{i}$; here, $h_{i,k}$ is the edge length
introduced in \eqref{eq:arcwise-mass}, and the vertex index is periodic when
$\curvepolygon{i}$ is closed.

We now define the mass-lumped vertex normal:
\begin{equation}\label{eq:nodal-normal-choices}
 \normalmass{m}{i,k}:=
 \begin{cases}
 \dfrac{h_{i,k-1}\normaldiscsec{i,k-1}{m}+h_{i,k}\normaldiscsec{i,k}{m}}
       {h_{i,k-1}+h_{i,k}}
 & \text{if }\curvepolygon{i}\text{ is closed},\\
 \normaldiscsec{i,1}{m}
 & \text{if }\curvepolygon{i}\text{ is open and }k=1,\\
 \dfrac{h_{i,k-1}\normaldiscsec{i,k-1}{m}+h_{i,k}\normaldiscsec{i,k}{m}}
       {h_{i,k-1}+h_{i,k}}
 & \text{if }\curvepolygon{i}\text{ is open and }1<k<N_i,\\
 \normaldiscsec{i,N_i-1}{m}
 & \text{if }\curvepolygon{i}\text{ is open and }k=N_i.
 \end{cases}
\end{equation}
Define the arc-wise flattening map
\begin{equation}
  \label{eq:flattening-map}
 \alpha(i,k):=\sum_{r=1}^{i-1}N_r+k,
 \qquad 1\leq i\leq I_S,\quad 1\leq k\leq N_i,
\end{equation}
so that $\alpha(i,k)\in\{1,\ldots,N_a\}$.
Let $B_n\in\R^{N_a\times2\totaln}$ be the nodal normal projection $B_n\bv{Z}\in\R^{N_a}$ of $\bv{Z}\in \R^{2\totaln}$ defined by
\[
(B_n \bv Z)_{\alpha(i,k)}
:=\bv Z_{\iota_i(k)}\cdot\normalmass{m}{i,k},
\qquad 1\leq i\leq I_S,\quad 1\leq k\leq N_i.
\]
For $\bvp\Phi\in(\femspaceSS{m}{})^{I_P}$, define
\[
 (\mathcal D_h^*\bvp\Phi)_i
 :=\sum_{p=0}^{I_P-1}m_{p,i}\,\Phi_p\lfloor_{\curvepolygonsec{m}_i},
 \qquad 1\leq i\leq I_S,
\]
whose coefficient matrix in the nodal bases is $\mathcal D_h^\top$.

Given vertices $\bv X^m$, one discrete step seeks
\[
 (\checmicalDiscVec{m+1},\kappa^{m+1},\bv X^{m+1})
 \in\femspaceWW{m}{g}\times\femspaceKK{m}{}\times\femspaceXX{m}{},
\]
with
\[
 v^{m+1}:=B_n\frac{\bv X^{m+1}-\bv X^m}{\Delta t}\in\femspaceKK{m}{},
\]
such that
\begin{align}
 \sum_{p=0}^{I_P-1}\productdual{\dtn{\curvepolygonsec{m}}\checmicalDisc{p}^{m+1}}{\Phi_p}{\curvepolygonsec{m}}
 - \productmass{v^{m+1}}{\mathcal D_h^*\bvp\Phi}{M_v}&=0
 &&\forall\bvp\Phi\in(\femspaceSS{m}{})^{I_P},\label{eq:weak-discrete-motion}\\
 \productmass{\mathcal D_h^*\bvp W^{m+1}-\kappa^{m+1}}{\chi}{M_v}&=0
 &&\forall\chi\in\femspaceKK{m}{},\label{eq:weak-discrete-gibbs-thomson}\\
 (\kappa^{m+1},B_n\bv\eta)_{M_v}
 +a_\gamma^m(\bv X^{m+1},\bv\eta)&=0
 &&\forall\bv\eta\in\femspaceXX{m}{},\label{eq:weak-discrete-curvature}
\end{align}
where
\[
  a^m_\gamma(\bv X, \bv \eta)
  :=\sum_{i=1}^{I_S}c_i\sum_{j\in\edgesDisc{m}{i}}\sum_{l=1}^{L_i}
  \frac{\bv\ell_j(\bv X)^{\top}\widetilde G_{i,l}\bv\ell_j(\bv \eta)}
       {s_{i,l}(\bv\ell_j^m)}
\]
with
\[
  s_{i,l}(\bv q):=\sqrt{\bv q^{\top}\widetilde G_{i,l}\bv q},
  \quad
  \widetilde G_{i,l}:=R_{\tfrac{\pi}{2}}^{\top}G_{i,l}R_{\tfrac{\pi}{2}},
  \quad
R_{\tfrac{\pi}{2}} := \begin{pmatrix}
  0 & -1\\
  1 & 0
\end{pmatrix}.
\]
Here and below, $R_\theta$ denotes the rotation matrix by the angle $\theta$.

For $\Xm\in\R^{2\totaln}$, define the discrete anisotropic interfacial energy by
\begin{equation}\label{eq:discrete-anisotropic-energy}
  \Eh(\Xm)
  :=\sum_{i=1}^{I_S}c_i\sum_{j\in\edgesDisc{m}{i}}\sum_{l=1}^{L_i}
  s_{i,l}\bigl(\bv\ell_j(\Xm)\bigr) = \int_{\curvepolygonsec{m}}\gamma(\normal)\,\dH{1}.
\end{equation}

\begin{remark}[Gauge equivalence under the exact-kernel assumption]
  Assume that $\ker\dtn{\curvepolygonsec{m}} = \operatorname{span}\{\bvp 1\}$.
  Then, we observe that the gauge condition \eqref{eq:gauge} implies the point-wise nodal zero-sum condition
\[
 \sum_{p=0}^{I_P-1}W_{p,j}^{m+1}=0,
 \qquad 1\leq j\leq N_{\mathrm{tot}}.
\]
Indeed, suppose that $\checmicalDiscVec{m+1}\in\femspaceWW{m}{g}$, take any $\phi\in\femspaceSS{m}{}$ and choose
$\bvp\Phi=(\phi,\ldots,\phi)\in(\femspaceSS{m}{})^{I_P}$ in \eqref{eq:weak-discrete-motion}.
By the definition of the discrete incidence action \eqref{eq:discrete-incidence-action} and
the zero-sum property of $M_{p,h}$ \eqref{eq:mph-sum-zero}, it follows that
\[
 \mathcal D_h^*\bvp\Phi
 =\sum_{p=0}^{I_P-1}M_{p,h}P_h\phi
 =\left(\sum_{p=0}^{I_P-1}M_{p,h}\right)P_h\phi
 =0 \in\R^{N_a},
\]
and hence the second term of \eqref{eq:weak-discrete-motion} vanishes. 
Therefore, we deduce from the linearity of $\dtn{\curvepolygonsec{m}}$ that
\begin{equation*}
 0 =\sum_{p=0}^{I_P-1}
 \productdual{\dtn{\curvepolygonsec{m}}W_p^{m+1}}{\phi}{\curvepolygonsec{m}}
 =\productdual{
 \dtn{\curvepolygonsec{m}}\left(\sum_{p=0}^{I_P-1}W_p^{m+1}\right)
 }{\phi}{\curvepolygonsec{m}}.
\end{equation*}
Since $\phi\in\femspaceSS{m}{}$ is arbitrary, it follows that
$$\dtn{\curvepolygonsec{m}}\left(\sum_{p=0}^{I_P-1}W_p^{m+1}\right)=0.$$
Hence, the exact-kernel assumption yields
\[
 \sum_{p=0}^{I_P-1}W_p^{m+1}=c\mathbf1.
\]
Since every component of $\widehat m_v$ is positive, we obtain
\[
 0=g^\top\bvp W^{m+1}
  =\widehat m_v^\top\sum_{p=0}^{I_P-1}W_p^{m+1}
  =c\,\widehat m_v^\top\mathbf1,
\]
and hence $c = 0$.
\end{remark}

With respect to the nodal bases of $\femspaceSS{m}{}$, $\femspaceKK{m}{}$, and
$\femspaceXX{m}{}$, we use the same symbols for finite element functions and its coefficient
vectors. Thus,
\[
 W_p^{m+1}\in\R^{\totaln},\qquad
 \kappa^{m+1}\in\R^{N_a},\qquad
 \bv X^m,\,\bv X^{m+1}\in\R^{2\totaln},
\]
and set
\[
 \checmicalDiscVec{m+1}
 :=(W_0^{m+1};\ldots;W_{I_P-1}^{m+1})
 \in\R^{I_P\totaln}.
\]
Let $A_\gamma^m\in\R^{2\totaln\times2\totaln}$ represent the stiffness form, i.e.,
\[
 \bv\eta^{\top}A_\gamma^m\bv X
 =a_\gamma^m(\bv X,\bv\eta)
 \qquad\text{for all}\qquad\bv X,\,\bv\eta\in\R^{2\totaln}.
\]
Then, the discrete weak formulation represented in \eqref{eq:weak-discrete-motion}, \eqref{eq:weak-discrete-gibbs-thomson}, and \eqref{eq:weak-discrete-curvature} is equivalent to
\begin{subequations}\label{eq:fully-discrete-scheme}
\begin{align}
  (\identitymatrix{I_P}\otimes \dtn{\curvepolygonsec{m}})\,\checmicalDiscVec{m+1} - \frac{1}{\dt}\, \mathcal D_h \Mv B_n (\bv X^{m+1} - \bv X^m) &= 0, \label{eq:F1}\\
  \mathcal D_h^{\top} \checmicalDiscVec{m+1} - \kappa^{m+1} &= 0, \label{eq:F2}\\
  B_n^{\top}\Mv\,\kappa^{m+1} + A^m_\gamma \bv X^{m+1} &= 0 \label{eq:E2}
\end{align}
together with the coefficient constraint
\begin{equation}
  \label{eq:G1}
 g^{\top}\checmicalDiscVec{m+1}=0.
\end{equation}
\end{subequations}
We can eliminate $\kappa^{m+1}$ from the system by substituting \eqref{eq:F2} into \eqref{eq:E2}.
Letting
\[
 \delta \bv X^{m+1} :=\bv X^{m+1}-\bv X^m,
 \qquad
 C_h^m:=\mathcal D_h\Mv B_n,
\]
the corresponding system has the symmetric block form:
\begin{equation}\label{eq:bordered-block-system}
\begin{pmatrix}
 \dt(\identitymatrix{I_P}\otimes\dtn{\curvepolygonsec{m}}) & -C_h^m & g\\
 -(C_h^m)^{\top} & -A_\gamma^m & 0\\
 g^{\top} & 0 & 0
\end{pmatrix}
\begin{pmatrix}
 \checmicalDiscVec{m+1}\\ \delta\bv X^{m+1}\\ \lambda
\end{pmatrix}
=
\begin{pmatrix}
 0\\ A_\gamma^m\bv X^m\\ 0
\end{pmatrix},
\end{equation}
where $\lambda\in\R$ is the rescaled gauge multiplier.

Before proving our first main result, we prepare an elementary lemma.
\begin{lemma}\label{lem:convex}
For $n\in\mathbb N$, let $G\in\R^{n\times n}$ be symmetric and let $s(\vec p) := \sqrt{\vec p^\top G\vec p}$ for $\vec p\in \R^n$.
If $G$ is positively semi-definite, then
\[
  \vec q^{\top}G(\vec q- \vec p) \ \ge\ \frac12\left(s(\vec q)^2 - s(\vec p)^2\right)
\]
holds for every $\vec p,\,\vec q\in\R^n$.
Meanwhile, if $G$ is positive definite and $p\neq0$, then
\[
  \frac{1}{2s(\vec p)}\left(s(\vec q)^2 - s(\vec p)^2\right) \ \ge\ s(q)-s(p)
\]
holds.
\end{lemma}
\begin{proof}
We compute
\begin{equation*}
  \begin{aligned}
\vec q^{\top}G(\vec q- \vec p) -\frac12\left(s(\vec q)^2 - s(\vec p)^2\right)
&= \frac12 \vec q^\top G\vec q - \vec q^\top G\vec p + \frac12 \vec p^\top G\vec p\\
&= \frac12 \vec q^\top G\vec q - \frac12 \vec q^\top G\vec p - \frac12 \vec q^\top G\vec p  + \frac12 \vec p^\top G\vec p\\
&= \frac12 \vec q^\top G(\vec q - \vec p) - \frac12 (\vec q - \vec p)^\top G\vec p\\
&= \frac12 \vec q^\top G(\vec q - \vec p) - \frac12 \vec p^\top G^\top (\vec q - \vec p)
= \frac12 (\vec q - \vec p)^\top G (\vec q - \vec p) \geq 0.
  \end{aligned}
\end{equation*}
Here, we have invoked the symmetric property  of $G$ to obtain the last identity,
and the last inequality follows from the positive semi-definiteness of $G$.
For the second inequality, we compute
\begin{equation*}
  \begin{aligned}
    0\leq (s(\vec q) - s(\vec p))^2
    = s(\vec q)^2 - 2s(\vec q)s(\vec p) + s(\vec p)^2
    = s(\vec q)^2 - s(\vec p)^2 -2s(\vec q)s(\vec p) + 2s(\vec p)^2,
  \end{aligned}
\end{equation*}
and thus
\[
2s(\vec p)(s(\vec q) - s(\vec p)) \leq s(\vec q)^2 - s(\vec p)^2.
\]
The proof is now complete.
\end{proof}

\begin{definition}
  The polygonal curve network $\curvepolygonsec{m}$ is said to be \textit{non-degenerate} provided that
$s_{i,l}(\bv{\ell}^m_j)>0$ holds for every $1\leq i\leq I_S$, $j\in\edgesDisc{m}{i}$, and $1\leq l\leq L_i$, so that $A^m_\gamma$ is defined.
\end{definition}

We are now in a position to state the main result of this paper.

\begin{theorem}\label{thm:main}
Assume that the hypothesis of Proposition~\ref{prop:discrete-dtn-structure} is valid, and
$\curvepolygonsec{m}$ is non-degenerate.
Then, any solution $(\checmicalDiscVec{m+1},\kappa^{m+1},\bv X^{m+1})$ of \eqref{eq:F1}--\eqref{eq:E2} satisfies, for every $\dt>0$,
\begin{equation}\label{eq:main}
  \Eh(\bv X^{m+1}) + \dt \sum_{p=0}^{I_P-1} (\checmicalDisc{p}^{m+1})^{\top}\dtn{\curvepolygonsec{m}}\checmicalDisc{p}^{m+1} \le \Eh(\bv X^{m})
\end{equation}
holds. In particular, $\Eh$ is non-increasing for every $\dt > 0$.
\end{theorem}
\vspace{-0.8em}
\begin{proof}
Let
\[
v_n^{m+1} = \frac{1}{\dt}\Bn \delta\bv{X}^{m+1},
\qquad \mathcal{Q}(\checmicalDiscVec{}) := \sum_{p=0}^{I_P-1} (\checmicalDisc{p})^{\top}\dtn{\curvepolygonsec{m}} \checmicalDisc{p}.
\]

Left-multiplying the coefficient equation
\eqref{eq:F1} by $(\checmicalDiscVec{m+1})^{\top}$ gives
\begin{equation}
  \label{eq:Dw-formula-1}
  \begin{aligned}
  0 &= \mathcal{Q}(\checmicalDiscVec{m+1}) - \frac1{\dt}\,(\checmicalDiscVec{m+1})^\top\mathcal D_h\Mv\Bn \delta \bv X^{m+1}\\
    &= \mathcal{Q}(\checmicalDiscVec{m+1}) - (\mathcal D_h^\top\checmicalDiscVec{m+1})^\top\Mv v_n^{m+1}
     = \mathcal{Q}(\checmicalDiscVec{m+1}) - (\kappa^{m+1})^\top\Mv v_n^{m+1}.
  \end{aligned}
\end{equation}
Here, we have invoked \eqref{eq:F2} to obtain the last identity.

Left-multiplying the coefficient equation
\eqref{eq:E2} by $(\delta\bv X^{m+1})^{\top}$ and using the symmetry of
$\Mv$, we compute
\begin{equation}
  \label{eq:Dw-formula-2}
  \begin{aligned}
  a^m_\gamma(\Xmp,\delta\Xmp)
  &= (\delta\Xmp)^{\top}A^m_\gamma\Xmp = -(\delta\Xmp)^{\top}\Bn^{\top}\Mv\kappa^{m+1}\\
  &= -\dt\, (v_n^{m+1})^\top \Mv\kappa^{m+1}
   = -\dt\,(\kappa^{m+1})^\top\Mv v_n^{m+1}.
  \end{aligned}
\end{equation}
We combine \eqref{eq:Dw-formula-1} and \eqref{eq:Dw-formula-2} to obtain
\begin{equation}\label{eq:pf:2}
  a^m_\gamma(\Xmp,\delta\Xmp) = -\dt\,\mathcal{Q}(\checmicalDiscVec{m+1}).
\end{equation}

We now recall the definition of the discrete anisotropic energy from \eqref{eq:discrete-anisotropic-energy}:
\[
\Eh(\Xm)=\sum_{i=1}^{I_S}c_i\sum_{j\in\edgesDisc{m}{i}}\sum_{l=1}^{L_i}
s_{i,l}(\bv\ell_j(\Xm)) .
\]
For each $1\leq i\leq I_S$, $j\in\edgesDisc{m}{i}$, and $1\leq l\leq L_i$,
put $\bv q = \bv{\ell}_j(\Xmp)$ and
$\bv p = \bv{\ell}_j(\Xm)$. Then, it is easily seen that $\bv q- \bv p = \bv{\ell}_j(\delta\Xmp)$ holds.
We now apply Lemma~\ref{lem:convex} with $G=\widetilde G_{i,l}$ and $s = s_{i,l}$ for $j\in\edgesDisc{m}{i}$
to derive
\[
  \frac{\bv q^{\top}\widetilde G_{i,l}(\bv q- \bv p)}{s_{i,l}(\bv p)}
  \ge \frac{s_{i,l}(\bv q)^2 - s_{i,l}(\bv p)^2}{2s_{i,l}(\bv p)}
  \ge s_{i,l}(\bv q)-s_{i,l}(\bv p),
\]
and thus
\[
  \frac{\bv \ell_j(\Xmp)^\top\widetilde{G}_{i,l}\bv \ell_j(\delta\Xmp)}{s_{i,l}(\bv\ell_j(\Xm))}
  \ge s_{i,l}(\bv \ell_j(\Xmp)) - s_{i,l}(\bv \ell_j(\Xm)).
\]
Multiplying the above inequality by $c_i$ and summing over $1\leq i \leq I_S$, $j\in\edgesDisc{m}{i}$, and $1\leq l\leq L_i$ gives
\begin{equation}\label{eq:pf:3}
  a^m_\gamma(\Xmp,\delta\Xmp) \ge \Eh(\Xmp)-\Eh(\Xm).
\end{equation}

Combining \eqref{eq:pf:2} and \eqref{eq:pf:3} yields
$$\Eh(\Xmp)-\Eh(\Xm) \le -\dt\,\mathcal{Q}(\checmicalDiscVec{m+1}),$$
which is nothing but \eqref{eq:main}, and this concludes the proof.
\end{proof}

\begin{remark}
  The inequality \eqref{eq:pf:3} is indeed already known for the class \eqref{eq:gamma} of anisotropies as stated in \cite[Lemma 102]{BGN20},
  although we have reproduced this estimate for the reader's convenience.
\end{remark}

We now show an area-preserving property of each phase for the fully discrete scheme \eqref{eq:F1}, \eqref{eq:F2}, and \eqref{eq:E2} in the velocity level.
\begin{proposition}\label{prop:cons}
For $0\leq p\leq I_P-1$, let $\bvp e_p\in\R^{I_P}$ be defined by $(\bvp e_p)_r := \delta_{pr}$.
Set
\[
 z_p:=\bvp e_p\otimes\mathbf1_{\totaln}\in\R^{I_P\totaln},
 \qquad
 d_p:=\mathcal D_h^{\top}z_p\in\R^{N_a}.
\]
Then, every solution of \eqref{eq:F1} satisfies $d_p^{\top}\Mv v_n^{m+1} = 0$ for
each phase $0\leq p\leq I_P-1$.
\end{proposition}
\begin{proof}
Left-multiplication of \eqref{eq:F1} by
$z_p^{\top}=\bvp e_p^{\top}\otimes\mathbf1_{\totaln}^{\top}\in \R^{1\times I_P\totaln}$ extracts the $p$-th phase block and gives
\begin{equation*}
  \begin{aligned}
0
&=z_p^{\top}
  (\identitymatrix{I_P}\otimes\dtn{\curvepolygonsec{m}})
  \checmicalDiscVec{m+1}
  -\frac1{\dt}z_p^{\top}\mathcal D_h\Mv\Bn \delta\Xmp\\
&=\mathbf1_{\totaln}^{\top}
  \dtn{\curvepolygonsec{m}}\checmicalDisc{p}^{m+1}
  -\frac1{\dt}
  (\mathcal D_h^{\top}z_p)^{\top}\Mv\Bn \delta\Xmp
=\bigl(\dtn{\curvepolygonsec{m}}\mathbf1_{\totaln}\bigr)^{\top}
  \checmicalDisc{p}^{m+1}
  -d_p^{\top}\Mv v_n^{m+1}.
  \end{aligned}
\end{equation*}
Here, the second equality uses the mixed-product rule for Kronecker products,
while the third uses the symmetry of $\dtn{\curvepolygonsec{m}}$.
We deduce from Proposition~\ref{prop:discrete-dtn-structure} that $\bvp 1_{\totaln} \in \ker\dtn{\curvepolygonsec{m}}$, and hence the first term of the right-hand side of the above identity vanishes.
This concludes the proof.
\end{proof}
We now show that the first variation of the discrete area of each phase
exactly vanishes, and that the resulting area defect equals the signed area of the
displacement polygon $\delta\Xmp$.
\begin{proposition}\label{prop:areadefect}
For $1\leq p\leq I_P-1$ and each region index $r\in\phase{p}$,
let $\regionDisc{r}(\bv X)$ denote its polygonal approximation of $\region{r}$.
Let $J_r^{\partial}$ be the number of closed polygonal boundary components and write
\[
 \partial\regionDisc{r}(\bv X)
 =\bigcup_{q=1}^{J_r^{\partial}}\Sigma_{r,q}^h(\bv X),
\]
and let $N_{r,q}^{\partial}$ be the number of vertices on
$\Sigma_{r,q}^h(\bv X)$.  We denote the corresponding ordered vertices,
obtained from the global geometry vector $\bv X$, by
\[
 \bv Y^{(r,q)}(\bv X)
 :=\bigl(\bv Y_j^{(r,q)}(\bv X)\bigr)_{j=1}^{N_{r,q}^{\partial}}.
\]
Each component is oriented so that $\regionDisc{r}(\bv X)$ lies locally on
its left; thus the outer boundary is counterclockwise and each hole boundary
is clockwise.  We define the signed area of $\regionDisc{r}(\bv X)$ by
\begin{equation}
  \label{eq:shoelace}
 \areaDiscRegion{r}(\bv X)
 :=\frac12\sum_{q=1}^{J_r^{\partial}}
 \sum_{j=1}^{N_{r,q}^{\partial}}
 \bv Y_j^{(r,q)}(\bv X)\times\bv Y_{j+1}^{(r,q)}(\bv X),
\end{equation}
where the vertex index is cyclic separately on each boundary component, i.e.,
$\bv Y_{N_{r,q}^{\partial}+1}^{(r,q)}(\bv X)=\bv Y_1^{(r,q)}(\bv X)$.
Here and below, $\bv a\times\bv b:=a_1b_2-a_2b_1$ denotes the scalar cross
product of $\bv a=(a_1,a_2)^\top$ and $\bv b=(b_1,b_2)^\top$ in $\R^2$.
Using this, the signed area of the $p$-th phase is defined by
\begin{equation}
  \label{eq:shoelace-p}
  \areaDisc{p}(\bv X) := \sum_{r\in\phase{p}}\areaDiscRegion{r}(\bv X).
\end{equation}
Then, every solution of \eqref{eq:F1} satisfies
\begin{equation}
  \label{eq:first-variation-area}
  \mathrm D\areaDisc{p}(\Xm)[\delta\Xmp] = -\dt\, d_p^\top\Mv v_n^{m+1} = 0.
\end{equation}
Consequently, it follows that
\begin{equation}\label{eq:areadefect}
  \areaDisc{p}(\Xmp) - \areaDisc{p}(\Xm) = \areaDisc{p}(\delta\Xmp).
\end{equation}
\end{proposition}
\begin{remark}
By the componentwise shoelace formula and the orientation specified above, it follows that
\begin{equation*}
 \mathcal L^2(\regionDisc{r}(\bv X)) = \areaDiscRegion{r}(\bv X).
\end{equation*}
\end{remark}
\begin{proof}[\textbf{Proof of Proposition~\ref{prop:areadefect}}]
For simplicity, set
\[
 \bv Y_j^{m,(r,q)}:=\bv Y_j^{(r,q)}(\Xm),
 \qquad
 \delta\bv Y_j^{m+1,(r,q)}:=\bv Y_j^{(r,q)}(\delta\Xmp).
\]
We differentiate \eqref{eq:shoelace} in the direction of $\delta\Xmp$ to obtain
\begin{equation*}
\begin{aligned}
 \mathrm D\areaDiscRegion{r}(\bv X)[\delta\Xmp]
 &:= \dde\areaDiscRegion{r}(\bv X + \varepsilon\delta\Xmp)\biggm|_{\varepsilon = 0}\\
 &=\frac12\sum_{q=1}^{J_r^{\partial}}
 \sum_{j=1}^{N_{r,q}^{\partial}}\left(
 \delta\bv Y_j^{m+1,(r,q)}\times\bv Y_{j+1}^{(r,q)}(\bv X)
 +\bv Y_j^{(r,q)}(\bv X)\times\delta\bv Y_{j+1}^{m+1,(r,q)}\right).
\end{aligned}
\end{equation*}
Evaluating this identity at $\bv X = \Xm$ together with the summation over $r\in\phase{p}$ gives
\begin{equation}
\begin{aligned}
 \mathrm D\areaDisc{p}(\Xm)[\delta\Xmp]
 &= \frac12\sum_{r\in\phase{p}}\sum_{q=1}^{J_r^{\partial}}\left(
 \sum_{j=1}^{N_{r,q}^{\partial}}\delta\bv Y_j^{m+1,(r,q)}\times\bv Y_{j+1}^{m,(r,q)}
 -\sum_{j=1}^{N_{r,q}^{\partial}}\delta\bv Y_j^{m+1,(r,q)}\times\bv Y_{j-1}^{m,(r,q)}\right)\\
 &=\frac12\sum_{r\in\phase{p}}\sum_{q=1}^{J_r^{\partial}}
 \sum_{j=1}^{N_{r,q}^{\partial}}
 \delta\bv Y_j^{m+1,(r,q)}\times
 \bigl(\bv Y_{j+1}^{m,(r,q)}-\bv Y_{j-1}^{m,(r,q)}\bigr),
\end{aligned}
\label{eq:phase-area-vertex-form}
\end{equation}
where the cyclic convention is applied separately for every pair $(r,q)$.
We next rewrite \eqref{eq:phase-area-vertex-form} in the curve-wise indexing
introduced in Section~\ref{sec:dio}.
At a non-junction vertex of $\curvepolygonsec{m}_i$,
the definition \eqref{eq:nodal-normal-choices} and the lumped weight
\eqref{eq:arcwise-mass} give
\begin{equation}\label{eq:area-gradient-normal}
 \mu_{i,k}\normalmass{m}{i,k}
 =\frac12\left(
 h_{i,k-1}\normaldiscsec{i,k-1}{m}
 +h_{i,k}\normaldiscsec{i,k}{m}
 \right).
\end{equation}
For every interface adjacent to phase $p$, we now introduce its phase-outward edge
normal by $\normaldiscsec{p;i,k}{m}$, and we observe that the jump convention in
\eqref{eq:Mp} gives
\begin{equation}\label{eq:phase-outward-normal}
 \normaldiscsec{p;i,k}{m}
 =-m_{p,i}\normaldiscsec{i,k}{m},
 \qquad m_{p,i}\in\{-1,1\},
\end{equation}
and $m_{p,i}=0$ when phase $p$ is not adjacent to $\curvepolygonsec{m}_i$.
Consequently, the phase-outward mass-lumped vertex normal satisfies
\begin{equation}\label{eq:phase-area-gradient-normal}
 \begin{aligned}
 \frac12\left(
 h_{i,k-1}\normaldiscsec{p;i,k-1}{m}
 +h_{i,k}\normaldiscsec{p;i,k}{m}
 \right)
 &=-m_{p,i}\mu_{i,k}\normalmass{m}{i,k}.
 \end{aligned}
\end{equation}
Suppose that vertex $j$ of the boundary component
$\Sigma_{r,q}^h$ in \eqref{eq:phase-area-vertex-form} corresponds to the
non-junction curve vertex $(i,k)$.  With the componentwise boundary orientation
introduced in \eqref{eq:shoelace}, we have
\[
 \bv Y_{j+1}^{m,(r,q)}-\bv Y_{j-1}^{m,(r,q)}
 =R_{\frac{\pi}{2}}
 \left(
 h_{i,k-1}\normaldiscsec{p;i,k-1}{m}
 +h_{i,k}\normaldiscsec{p;i,k}{m}
 \right).
\]
On noting that
$\bv a\times\bv b=\bv a^{\top}R_{\frac{\pi}{2}}^{\top}\bv b$ and
$R_{\frac{\pi}{2}}^{\top}R_{\frac{\pi}{2}}=I_2$, it follows from
\eqref{eq:phase-area-gradient-normal} that
\begin{equation}
\begin{aligned}
 \frac12\delta\bv Y_j^{m+1,(r,q)}\times
 \left(\bv Y_{j+1}^{m,(r,q)}-\bv Y_{j-1}^{m,(r,q)}\right)
 &= \frac12(\delta\bv Y_j^{m+1,(r,q)})^{\top}
 R_{\frac{\pi}{2}}^{\top}R_{\frac{\pi}{2}}
 \left(
 h_{i,k-1}\normaldiscsec{p;i,k-1}{m}
 +h_{i,k}\normaldiscsec{p;i,k}{m}
 \right) \\
 &=-m_{p,i}\mu_{i,k}\,
 \delta\bv X_{\iota_i(k)}^{m+1}\cdot\normalmass{m}{i,k}.
 \label{eq:phase-area-local-contribution}
\end{aligned}
\end{equation}
We can easily perform a similar computation to obtain the formula \eqref{eq:phase-area-local-contribution} at a triple junction.
Indeed, the junction is an endpoint of the two curves which are part of
$\partial\regionDisc{r}$ that meet there, so \eqref{eq:arcwise-mass} and
\eqref{eq:nodal-normal-choices} give them the weight $\tfrac12h$ and the adjacent edge
normal, while the third incident curve separates the two phases other than $p$ and
therefore carries $m_{p,i}=0$.
Therefore, we deduce from \eqref{eq:phase-area-vertex-form} that
\begin{equation}
 \label{eq:phase-area-matrix-form}
\begin{aligned}
 \mathrm D\areaDisc{p}(\Xm)[\delta\Xmp]
 &=-\sum_{i=1}^{I_S}\sum_{k=1}^{N_i}
 m_{p,i}\,\mu_{i,k}\,
 \delta\bv X_{\iota_i(k)}^{m+1}\cdot\normalmass{m}{i,k}
 =-d_p^{\top}\Mv\Bn\delta\Xmp,
\end{aligned}
\end{equation}
where on recalling the arc-wise flattening map $\alpha(i,k)$ defined by \eqref{eq:flattening-map},
\[
 (d_p)_{\alpha(i,k)}=m_{p,i},\qquad
 (\Mv)_{\alpha(i,k),\alpha(i,k)}=\mu_{i,k},\qquad
 (\Bn\delta\Xmp)_{\alpha(i,k)}
 =\delta\bv X_{\iota_i(k)}^{m+1}\cdot\normalmass{m}{i,k}.
\]
We deduce from the relation $\Bn\delta\Xmp=\dt\,v_n^{m+1}$,
and Proposition~\ref{prop:cons},
the formula \eqref{eq:phase-area-matrix-form} yields
\begin{equation*}
 \mathrm D\areaDisc{p}(\Xm)[\delta\Xmp]
 =-\dt\,d_p^{\top}\Mv v_n^{m+1}
 =0,
\end{equation*}
which proves \eqref{eq:first-variation-area}.
Moreover, since $\areaDisc{p}$ is quadratic, we have
\begin{equation*}
 \areaDisc{p}(\Xmp)-\areaDisc{p}(\Xm)
 =\mathrm D\areaDisc{p}(\Xm)[\delta\Xmp]
   +\areaDisc{p}(\delta\Xmp)
 =\areaDisc{p}(\delta\Xmp),
\end{equation*}
which proves \eqref{eq:areadefect}.
The proof is now complete.
\end{proof}

As a conclusion of this section, we shall characterize curve-network conditions for the solvability of the system \eqref{eq:F1}, \eqref{eq:F2}, and \eqref{eq:E2}.
\begin{definition}[Solvability-admissible curve network]
  \label{dfn:admissible-curve-network}
A non-degenerate polygonal curve network $\curvepolygon{}$ is called
\emph{solvability-admissible} if the following conditions hold:
\begin{enumerate}\itemsep0pt
  \item[(i)] every curve has a non-junction vertex, and the corresponding
  lumped masses and nodal normals defined in \eqref{eq:nodal-normal-choices}
  are nonzero; equivalently, for every $1\leq i\leq I_S$,
  \begin{equation}\label{eq:solv-admissible-nonjunction}
    \mathcal N_i^{\circ}
    :=\left\{\bv X_{i,k}:1\leq k\leq N_i\right\}
      \setminus\bigcup_{s=1}^{I_T}\junctionDisc{s}\neq\emptyset,
    \qquad
    \mu_{i,k}\neq0,\qquad\normalmass{m}{i,k}\neq\bv{0}
    \quad\text{for every }\bv X_{i,k}\in\mathcal N_i^{\circ}.
  \end{equation}
  \item[(ii)] on every connected component $C$ of the curve network, the
  nodal normals at non-junction vertices span $\R^2$, that is,
  \begin{equation}\label{eq:solv-admissible-normal-span}
    \operatorname{span}\left\{\normalmass{m}{i,k}\biggm|\,
      \bv X_{i,k}\in C\setminus\bigcup_{s=1}^{I_T}\junctionDisc{s}\right\}
    =\R^2.
  \end{equation}
  \item[(iii)] Let the finite undirected phase-adjacency graph $G_P=(V_P,E_P)$ defined by
        \[
          V_P :=\{0,\ldots,I_P-1\},
        \]
        and $\{p,r\}\in E_P$ if and only if $p\neq r$ and there exists a curve $\curvepolygon{i}$ separating the $p$-th and $r$-th phases.
        Then, the graph $G_P$ is connected. In other words, for every $p,\,r\in V_P$ with $p\neq r$, there exists a sequence $\{p_0,p_1,\ldots,p_k\}$
        for which $p_0 = p$, $p_k = r$ and $\{p_{j-1}, p_j\}\in E_P$ for every $1\leq j\leq k$.
\end{enumerate}
\end{definition}

\begin{theorem}[network unique solvability]\label{prop:solv}
Assume that $\curvepolygonsec{m}$ satisfies the hypothesis of
Proposition~\ref{prop:discrete-dtn-structure} and is solvability-admissible.
Impose the gauge $g^{\top}\checmicalDiscVec{m+1}=0$, where
$g^{\top}(\mathbf1_{I_P}\otimes\mathbf1_{\totaln})\neq0$.
Then, for every $\dt>0$, the gauge-bordered system associated with
\eqref{eq:F1}--\eqref{eq:E2} has exactly one solution $(\checmicalDiscVec{m+1},\kappa^{m+1},\bv X^{m+1})$.
\end{theorem}
\begin{proof}
Write $q:=\mathbf1_{I_P}\otimes\mathbf1_{\totaln} = \mathbf{1}_{I_P\totaln}\in\R^{I_P\totaln}$.
Assume that $(\checmicalDiscVec{},\kappa, \delta\bv X, \lambda)$ solves the following homogeneous system:
\begin{subequations}\label{eq:homogeneous-system}
\begin{align}
 (\identitymatrix{I_P}\otimes\dtn{\curvepolygonsec{m}})\checmicalDiscVec{}
 -\dt^{-1}\mathcal D_h\Mv\Bn\delta \bv X+\lambda g&=0,\label{eq:homogeneous-system-a}\\
 \mathcal D_h^{\top}\checmicalDiscVec{}-\kappa&=0,\label{eq:homogeneous-system-b}\\
 \Bn^{\top}\Mv\kappa+A^m_\gamma\delta \bv X&=0,\label{eq:homogeneous-system-c}\\
 g^{\top}\checmicalDiscVec{}&=0.\label{eq:homogeneous-system-d}
\end{align}
\end{subequations}
We see from Proposition~\ref{prop:discrete-dtn-structure} that $\dtn{\curvepolygonsec{m}}$ is symmetric and $\mathbf1_{\totaln}\in\ker\dtn{\curvepolygonsec{m}}$, and thus,
\[
\vectorunit{\totaln}^\top\dtn{\curvepolygonsec{m}} = ((\dtn{\curvepolygonsec{m}})^\top\vectorunit{\totaln})^{\top} = (\dtn{\curvepolygonsec{m}}\vectorunit{\totaln})^\top = \vectorzero{\totaln}^\top.
\]
Hence, we have $q^\top(\identitymatrix{I_P}\otimes\dtn{\curvepolygonsec{m}}) = 0$.
Meanwhile, by \eqref{eq:discrete-incidence-action} and \eqref{eq:mph-sum-zero},
we have $\mathcal D_h^\top q=0$, and hence $q^\top\mathcal D_h=0$.
Left-multiplying \eqref{eq:homogeneous-system-a} by $q^{\top}$ gives
\[
  0=q^{\top}\left[
  (\identitymatrix{I_P}\otimes\dtn{\curvepolygonsec{m}})\checmicalDiscVec{}
  -\dt^{-1}\mathcal D_h\Mv\Bn\delta\bv X+\lambda g\right]
  =\lambda q^{\top}g.
\]
Since $q^{\top}g=g^{\top}q\neq0$ by the gauge condition \eqref{eq:homogeneous-system-d}, it follows that $\lambda=0$.

We next left-multiply \eqref{eq:homogeneous-system-a} and \eqref{eq:homogeneous-system-c} by
$(\checmicalDiscVec{})^{\top}$ and $(\delta\bv X)^{\top}$, respectively,
then we obtain
\begin{equation}
  \label{eq:uniqueness-1}
\begin{aligned}
 0
 &=\sum_{p=0}^{I_P-1}(\checmicalDisc{p})^{\top}
   \dtn{\curvepolygonsec{m}}\checmicalDisc{p}
   -\dt^{-1}(\checmicalDiscVec{})^{\top}
     \mathcal D_h\Mv\Bn\delta\bv X,\\
 (\delta\bv X)^{\top}A^m_\gamma\delta\bv X
 &=-(\delta\bv X)^{\top}\Bn^{\top}\Mv\kappa
 = -(\delta\bv X)^\top \Bn^{\top}\Mv\mathcal D_h^\top\checmicalDiscVec{}
  =-(\checmicalDiscVec{})^{\top}\mathcal D_h\Mv\Bn\delta\bv X,
\end{aligned}
\end{equation}
where we have invoked \eqref{eq:homogeneous-system-b} and the symmetry of $\Mv$.
Combining the two equations in \eqref{eq:uniqueness-1} yields
\begin{equation}\label{eq:network-solv-energy}
 \sum_{p=0}^{I_P-1}(\checmicalDisc{p})^{\top}
 \dtn{\curvepolygonsec{m}}\checmicalDisc{p}
 +\dt^{-1}(\delta\bv X)^{\top}A^m_\gamma\delta\bv X=0.
\end{equation}
We can observe that two terms are nonnegative in the left-hand side of \eqref{eq:network-solv-energy}.
Indeed, Proposition~\ref{prop:discrete-dtn-structure} gives this property for the first
sum, while the definition of $a_\gamma^m$ gives
\[
 (\delta\bv X)^{\top}A^m_\gamma\delta\bv X
 =\sum_{i=1}^{I_S}c_i\sum_{j\in\edgesDisc{m}{i}}\sum_{l=1}^{L_i}
 \frac{\bv\ell_j(\delta\bv X)^{\top}\widetilde G_{i,l}
       \bv\ell_j(\delta\bv X)}
      {s_{i,l}(\bv\ell_j^m)}\geq0.
\]
Consequently, \eqref{eq:network-solv-energy} implies that
\begin{equation}
  \label{eq:uniqueness-2}
  \begin{aligned}
 (\checmicalDisc{p})^{\top}\dtn{\curvepolygonsec{m}}
 \checmicalDisc{p}&=0,\qquad 0\leq \forall p<I_P,\\
 \bv\ell_j(\delta\bv X) &=0,
 \qquad\forall j\in\edgesDisc{m}{i},\ 1\leq \forall i\leq I_S.
  \end{aligned}
\end{equation}
Since $\dtn{\curvepolygonsec{m}}$ is symmetric and positive semi-definite,
the first equation of \eqref{eq:uniqueness-2} implies that $\checmicalDisc{p}\in\ker\dtn{\curvepolygonsec{m}}$.
Therefore, Proposition~\ref{prop:discrete-dtn-structure} gives
\begin{equation}
  \label{eq:uniqueness-3}
 \checmicalDisc{p}=c_p\vectorunit{\totaln},
 \qquad 0\leq \forall p<I_P.
\end{equation}
For every connected component $C$ of the curve network,
the second equation of \eqref{eq:uniqueness-2} implies that
\begin{equation}\label{eq:componentwise-translation}
 \delta\bv X\lfloor_C\equiv \bv t_C\in\R^2.
\end{equation}
We invoke $\lambda=0$ and \eqref{eq:uniqueness-3},
and then it follows that for every non-junction row $(p,\iota_i(k))$,
\begin{equation}\label{eq:translation-kernel}
 0=\bigl[\mathcal D_h\Mv\Bn\delta\bv X\bigr]_{p,\iota_i(k)}
  =m_{p,i}\mu_{i,k}\,\bv t_C\cdot\normalmass{m}{i,k}.
\end{equation}
Hence, we have
\begin{equation}\label{eq:translation-vanishes}
 \bv t_C\perp
 \operatorname{span}\!\left\{\normalmass{m}{i,k}\biggm|
   \bv X_{i,k}\in C\ \text{is not a junction}\right\}
 =\R^2.
\end{equation}
Thus, due to the conditions \eqref{eq:solv-admissible-nonjunction} and \eqref{eq:solv-admissible-normal-span},
it must hold that $\bv t_C = \bv 0$ for every connected component $C$, proving $\delta\bv X = 0$.

We deduce from \eqref{eq:homogeneous-system-c} that $\Bn^{\top}\Mv\kappa=0$.
For each curve $\curvepolygonsec{m}_i$, due to the condition \textup{(i)} in
Definition~\ref{dfn:admissible-curve-network}, there exists a non-junction vertex
$\bv X_{i,k}\in\mathcal N_i^{\circ}$. For such $k$, we have
\begin{equation}\label{eq:phase-constants-propagate}
 0=\bigl[\Bn^{\top}\Mv\kappa\bigr]_{\iota_i(k)}
   =\mu_{i,k}\kappa_{i,k}\normalmass{m}{i,k}.
\end{equation}
Since $\mu_{i,k}\neq0$ and $\normalmass{m}{i,k}\neq\bv0$, this gives
$\kappa_{i,k}=0$ at the chosen non-junction vertex.

Suppose that $\curvepolygonsec{m}_i$ separates phases $\phase{p}$ and $\phase{r}$.
Since \eqref{eq:uniqueness-3} makes every phase trace constant,
\eqref{eq:homogeneous-system-b} and the incidence convention give
\[
 \kappa_i
 =\bigl(m_{p,i}c_p+m_{r,i}c_r\bigr)\mathbf1_{N_i}
 =m_{r,i}(c_r-c_p)\mathbf1_{N_i},
 \qquad m_{r,i}=-m_{p,i}\in\{-1,1\}.
\]
Evaluating this identity at the chosen non-junction vertex, where
$\kappa_{i,k}=0$, yields $c_r=c_p$ and therefore $\kappa_i=0$ on the whole
curve.  Condition \textup{(iii)} in Definition~\ref{dfn:admissible-curve-network} stresses that
\[
 c_0=\cdots=c_{I_P-1}=:c \in\R,
\]
and we deduce from \eqref{eq:uniqueness-3} that $\checmicalDiscVec{} = c\vectorunit{I_P\totaln}$.

The gauge condition \eqref{eq:homogeneous-system-d} gives
\begin{equation}\label{eq:phase-gauge-kernel}
 0=g^{\top}\checmicalDiscVec{}=c\,g^{\top}q,
\end{equation}
which yields $c = 0$, and thus $\checmicalDiscVec{}\equiv\vectorzero{I_P\totaln}$ on recalling that $g^\top q \neq 0$.

Finally, we have shown that
\[
 (\checmicalDiscVec{},\kappa,\delta\bv X,\lambda)=(\vectorzero{I_P\totaln},\vectorzero{N_a},\vectorzero{2\totaln},0).
\]
The system \eqref{eq:homogeneous-system-a}--\eqref{eq:homogeneous-system-d} is linear and square, namely
the unknown $(\checmicalDiscVec{},\kappa,\delta\bv X,\lambda)$ and the equations both
number $I_P\totaln+N_a+2\totaln+1$.
Its coefficient matrix is therefore nonsingular, and
the corresponding inhomogeneous system has exactly one solution, from which
$\bv X^{m+1}=\bv X^m+\delta\bv X^{m+1}$ is determined as well.
The proof is now complete.
\end{proof}

\section{Numerical experiments}\label{sec:numerics}
\newcommand{\adfmOne}{2.060\times10^{-15}}
\newcommand{\adfmOneFlux}{2.044\times10^{-15}}
\newcommand{\adfmOneRatio}{1.01}
\newcommand{\adfmOneDrift}{8.940\times10^{-5}}
\newcommand{\adfmTen}{1.487\times10^{-15}}
\newcommand{\adfmTenFlux}{1.569\times10^{-15}}
\newcommand{\adfmTenRatio}{0.95}
\newcommand{\adfmTenDrift}{8.815\times10^{-4}}
\newcommand{\adfmHun}{1.517\times10^{-15}}
\newcommand{\adfmHunFlux}{1.627\times10^{-15}}
\newcommand{\adfmHunRatio}{0.93}
\newcommand{\adfmHunDrift}{3.809\times10^{-3}}
\newcommand{\adfmThou}{3.235\times10^{-15}}
\newcommand{\adfmThouFlux}{3.203\times10^{-15}}
\newcommand{\adfmThouRatio}{1.01}
\newcommand{\adfmThouDrift}{9.742\times10^{-3}}
\newcommand{\adfmMax}{3.235\times10^{-15}}
\newcommand{\adfmMaxRel}{4.119\times10^{-15}}
\newcommand{\adfmRatioMax}{1.01}
\newcommand{\adfmDriftMax}{9.742\times10^{-3}}
\newcommand{\adfmRuns}{4}

We use the square-root class \eqref{eq:gamma} of anisotropies
which may differ on each interface; precisely speaking, we set
\[
  c_i=c,
  \qquad
  G_{i,l}=R_{\varphi_l}\operatorname{diag}(1,\delta^2)
  R_{\varphi_l}^\top,
  \qquad \varphi_l=\frac{l\pi}{L_i},
\]
where for each label of anisotropy in Table~\ref{tab:aniso},
the parameters $\delta$ and $c$ are determined so that the following conditions are fulfilled.
\begin{equation}\label{eq:pin}
  \frac{\max_\theta\gamma}{\min_\theta\gamma}
  =\text{target ratio},
  \qquad
  \frac{\max_\theta\gamma+\min_\theta\gamma}{2}=1.
\end{equation}
\begin{table}[!ht]
\centering
\small
\begin{tabular}{llccll}
\toprule
Label & Wulff shape & $L$ & Ratio & $\delta$ & $c$ \\
\midrule
$\gamma_E$ & ellipse & $1$ & $1.30$ & $0.76923076923076916$ & $1.1304347826086958$ \\
$\gamma_S$ & square & $2$ & $1.30$ & $0.092500733601242183$ & $0.79594016795295963$ \\
$\gamma_H$ & hexagon & $3$ & $1.11$ & $0.075828991533126139$ & $0.52381703438571303$ \\
\bottomrule
\end{tabular}
\caption{Three anisotropies used in the numerical experiments and associated parameters.}
\label{tab:aniso}
\end{table}

In every experiment of this section, the polygonal boundary of each bounded region has
diameter at most $1.65$ at every time level, so that Lemma~\ref{lem:diam-cap} implies that
the logarithmic capacity on those boundaries is less than $1$, and the hypothesis of
Proposition~\ref{prop:discrete-dtn-structure} is satisfied.

\subsection{Energy dissipation}\label{subsec:dtsweep}
Here, we confirm that the proposed fully discrete scheme \eqref{eq:fully-discrete-scheme} has the energy dissipation property
regardless the value of time step is set large as shown in Theorem~\ref{thm:main}.
In this experiment, we suppose that two regions occupy distinct phases.
The anisotropy $\gamma_S$ are assigned to both interfaces.
We set
\[
\begin{aligned}
  I_P&=I_R=3, &I_S = 2,\ I_T = 0,\ N&=32,\\
  \dt_0&=\frac{0.005}{1024},&
  \dt&\in\{1,10,100,1000\}\dt_0,\qquad T=9000\dt_0 .
\end{aligned}
\]
We summarize the maximum of the variation of the energies of consecutive two discrete curve networks.
Namely, we set $$\Delta E_{\max}:=\max_{0\leq m< M}(E_h(\Xmp)-E_h(\Xm)),$$
and we report its values with respect to $\dt$ as follows.

\begin{center}
\begin{tabular}{lr}
\toprule
$\dt$ & $\Delta E_{\max}$ \\
\midrule
$\dt_0$     & $-1.60\times10^{-12}$ \\
$10\dt_0$   & $-1.15\times10^{-8}$ \\
$100\dt_0$  & $-7.22\times10^{-7}$ \\
$1000\dt_0$ & $-1.19\times10^{-3}$ \\
\bottomrule
\end{tabular}
\end{center}

According to the outcome of this experiment, the discrete energy $E_h(\Xm)$ is non-increasing at any time and its step size.
Moreover, each distinct component evolves to a square-like curves which is consistent with the Wulff shape of the anisotropy (see Figure~\ref{fig:m6:a}).
\begin{figure}[!ht]
  \centering
  \includegraphics[width=\textwidth]{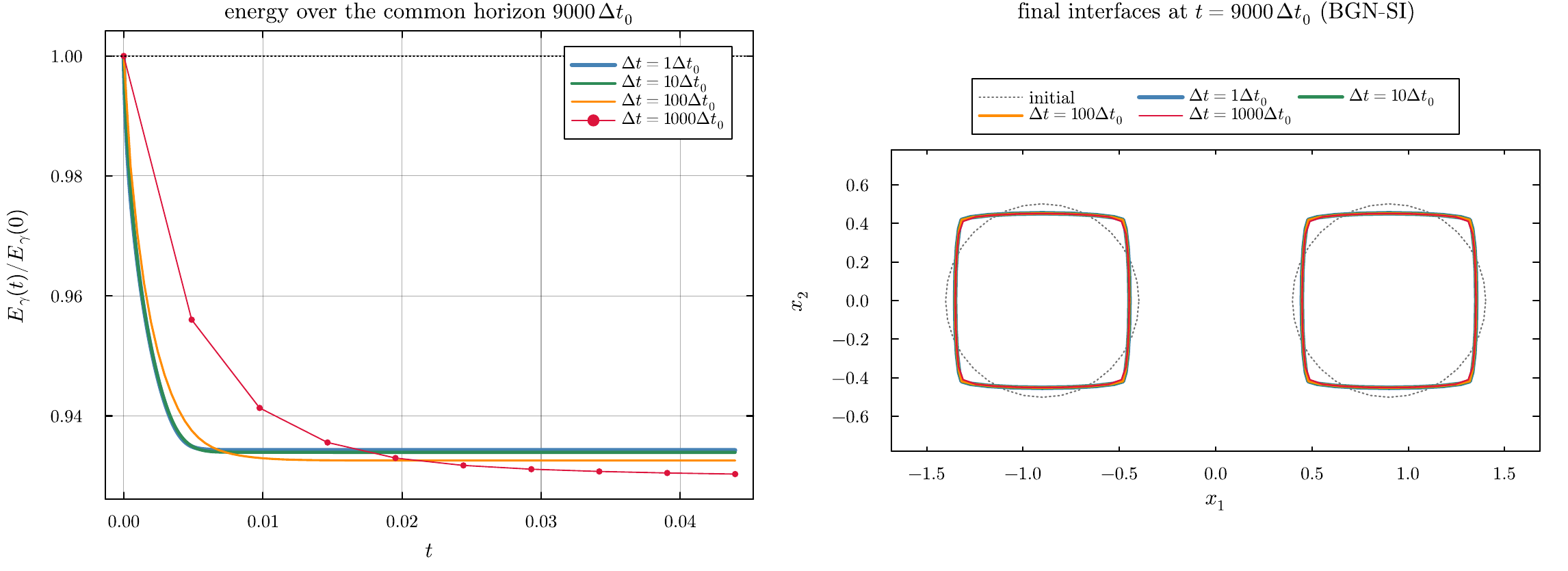}
  \caption{Time-step sweep: evolution of relative energy (left) and two components (right).}
  \label{fig:m6:a}
\end{figure}

\subsection{Area-defect identity}\label{subsec:areadefect}
Here, for the same experiment carried out in Section~\ref{subsec:dtsweep},
we measure the area defect defined by
\[
  r:=\max_{\substack{0\leq m< M\\1\leq p<I_P}}
  \left|\areaDisc{p}(\Xmp)-\areaDisc{p}(\Xm)-\areaDisc{p}(\delta \Xmp)\right|,
\]
where $\delta\Xmp :=\Xmp-\Xm$, and we recall from \eqref{eq:shoelace-p} that
$\areaDisc{p}(\Xm)$ is the signed polygonal area form of phase $p$ at time step $m$.
We compare it with
\[
  q:=\dt\max_{\substack{0\leq m< M\\1\leq p < I_P}}
  \left|d_p^\top\Mv v_n^{m+1}\right|,
\]
which is shown to be the first variation of the discrete area (see the formula \eqref{eq:first-variation-area}).
Define the terminal area drift as the relative error of the discrete area at the final time step $M$ compared to one at the initial data.
Precisely, we let
  \[
  d^M_A:=
  \max_{\substack{1\leq p<I_P}}
  \frac{\left|\areaDisc{p}(\bv X^M)-\areaDisc{p}(\bv X^0)\right|}
  {\left|\areaDisc{p}(\bv X^0)\right|}.
  \]
By \eqref{eq:first-variation-area}, the quantities $r$ and $q$ agree in exact
arithmetic, so that their ratio in Table~\ref{tab:m6:b}
measures the floating-point error alone.
Moreover, Proposition~\ref{prop:areadefect} asserts that their common value vanishes, and it is
found at the round-off level in every case.
Meanwhile, the terminal area drift $d^M_A$ is higher than them, although this error is also expected due to the residual $\areaDisc{p}(\delta\Xmp)$.

\begin{table}[!ht]
  \centering
  {\begin{tabular}{lrrrrr}
\toprule
$\dt$ & $T / \dt$ & $r$ & $q$ & $r / q$ & $d^M_A$ \\
\midrule
$\dt_0$ & 9000 & $2.060\times10^{-15}$ & $2.044\times10^{-15}$ & $1.01$ & $8.940\times10^{-5}$ \\
$10\dt_0$ & 900 & $1.487\times10^{-15}$ & $1.569\times10^{-15}$ & $0.95$ & $8.815\times10^{-4}$ \\
$100\dt_0$ & 90 & $1.517\times10^{-15}$ & $1.627\times10^{-15}$ & $0.93$ & $3.809\times10^{-3}$ \\
$1000\dt_0$ & 9 & $3.235\times10^{-15}$ & $3.203\times10^{-15}$ & $1.01$ & $9.742\times10^{-3}$ \\
\bottomrule
\end{tabular}
}
  \caption{Time-step sweep: area-defect residual, constraint residual, and terminal area drift.}
  \label{tab:m6:b}
\end{table}

\subsection{Coarsening with different anisotropy}\label{subsec:quadost}
Here, we carry out a numerical experiment for a different curve network composed of four components;
two of it share a common anisotropy $\gamma_1$, whereas other two share a different anisotropy $\gamma_2$.
In this experiment, we set $I_P = 3$, $I_R = 5$, $I_S = 4$, and $I_T = 0$.
The initial data of each two components is the union of two circles whose radii are different.
In this experiment, we set
\[
  N=64,\qquad \dt=100\dt_0,\qquad
  (\gamma_1,\gamma_2)=(\gamma_S,\gamma_H).
\]
Figure~\ref{fig:m6:c} shows the interface evolution and the corresponding diagnostics,
that is, the evolution of the discrete energy and discrete area of each region in the curve network.
We observe that the droplets change their shape to squared and hexagonal ones which correspond to the assigned anisotropies.
Moreover, the smaller drops are going to vanish, whereas the areas of the larger ones increase gradually.
We terminate the iteration at the first step at which the discrete area
$\areaDisc{}$ of some component falls below one fifth of its initial value, so that the
run ends strictly before any topological change. This occurs at step $454$ of the
planned $3000$, that is $t=0.2217$, and is triggered by the discrete area of the small $\gamma_H$ drop
($9.989\times10^{-2}$ against the threshold $1.005\times10^{-1}$); the discrete area of the small $\gamma_S$
drop is then still at $21.7\%$ of its initial area.
The two phase totals are flat on the scale of Figure~\ref{fig:m6:c}; at the
terminal step, their relative drifts are $-3.33\times10^{-3}$ for the $\gamma_S$ phase and
$-7.96\times10^{-4}$ for the $\gamma_H$ phase.

\begin{figure}[!ht]
  \centering
  \includegraphics[width=\textwidth]{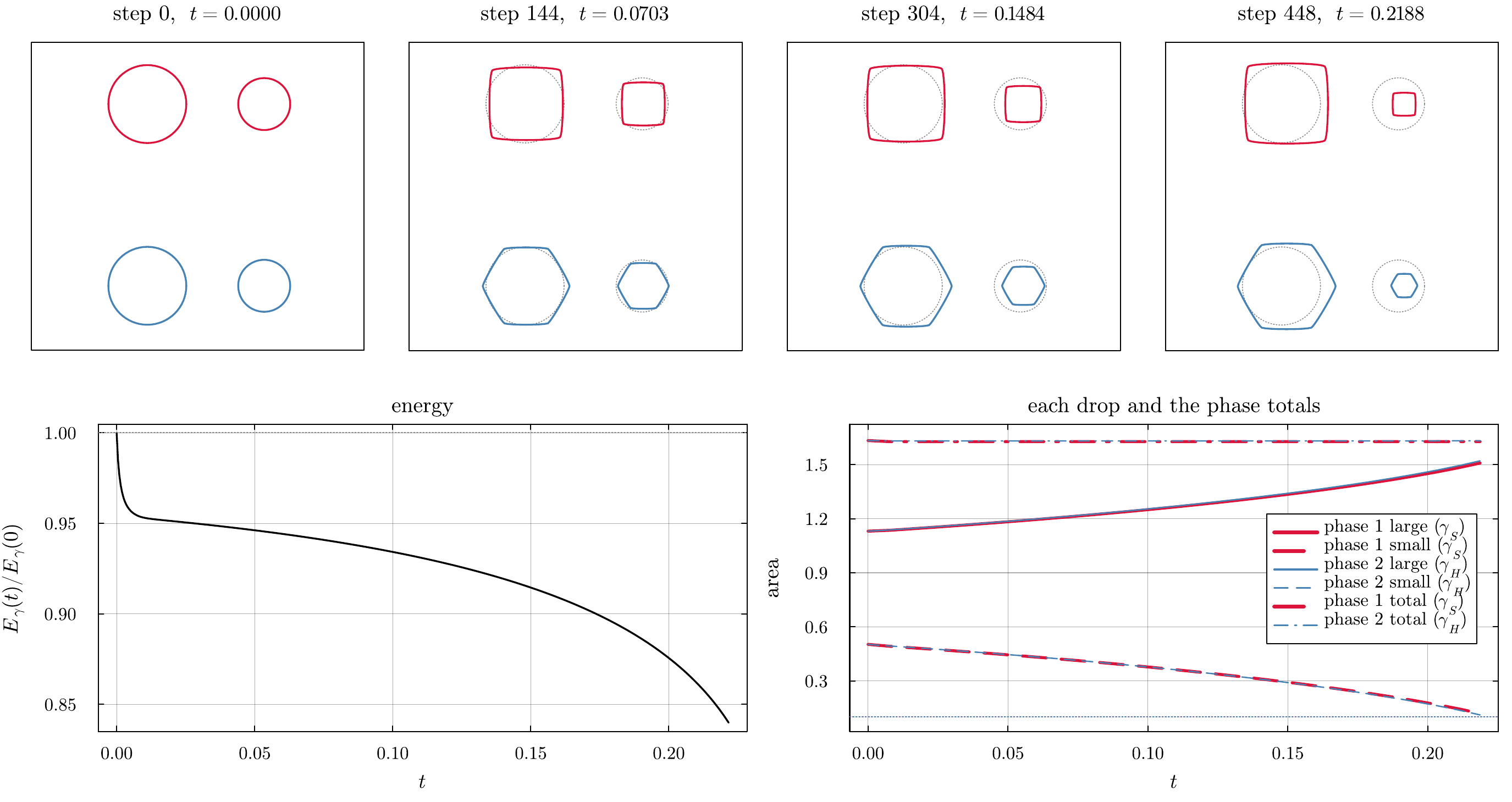}
  \caption{Three-phase coarsening: evolution of the curve network, relative energy, and the area of each
  component together with the total area of each phase; the dotted line is the stopping threshold of the area.}
  \label{fig:m6:c}
\end{figure}

\subsection{Triple-junction network carrying three different anisotropies}\label{sec:m4}

The perturbed triple bubble has six curves, four phases, four triple junctions,
namely $I_P = I_R = 4$, $I_S = 6$, and $I_T = 4$.
The number of edges on each curve is set to $(47,47,47,17,17,17)$.
The anisotropies assigned to the six curves are
$(\gamma_E, \gamma_S,\gamma_H, \gamma_S, \gamma_H, \gamma_E)$,
so that every triple junction meets three distinct anisotropies.
Figure~\ref{fig:m7:d} shows the interface evolution, the variation of anisotropic energy against the initial energy,
and the change of each enclosed area relative to its initial value.
We implement $2000$ iterations with $\dt=2.5\times10^{-6}$ and observe that $\Delta E_{\max}=-3.62\times10^{-7}$,
and thus the anisotropic surface energy is non-increasing.
The bounded-phase area drifts are $(1.07\times10^{-6},-7.89\times10^{-4},-2.59\times10^{-4})$.

\begin{figure}[!t]
  \centering
  \includegraphics[width=\textwidth]{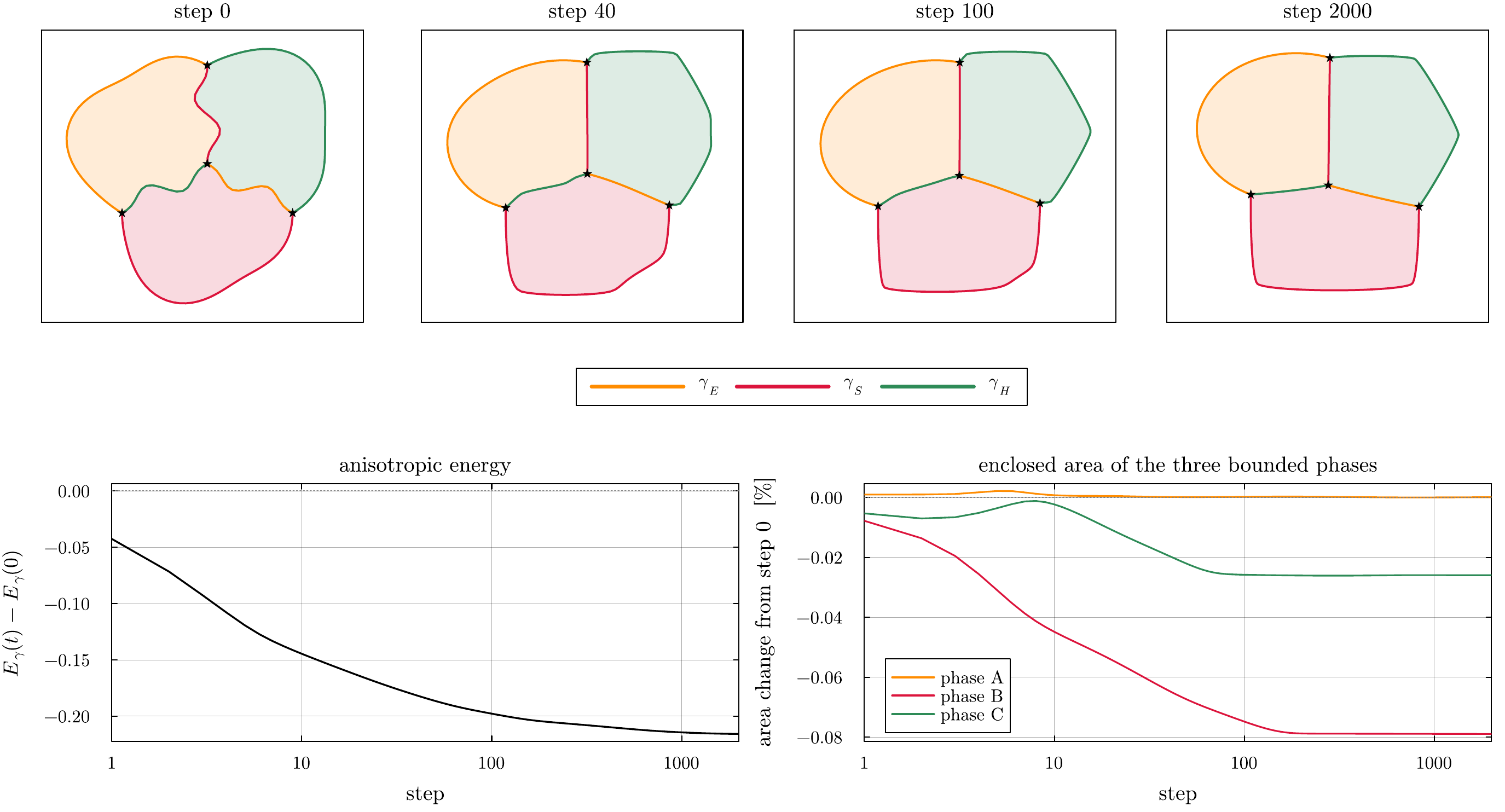}
  \caption{Mixed-anisotropy triple bubble: evolution of the curve network,
  variation of anisotropic energy, and bounded-phase areas.}
  \label{fig:m7:d}
\end{figure}

We also measure the discrete Herring residual defined by

\begin{equation}\label{eq:strong:RH}
  \RH(\junctionDiscSec{m}{k})
  :=\left|\sum_{\ell=1}^{3}\bv \xi_{s^k_\ell}(\widehat{\bv\nu^m_{k,\ell}})^\perp\right|
  \sentence{for} 1\leq k\leq I_T,
\end{equation}
where, letting $\bv Q^m_{k,\ell}$ be the vertex of
$\curvepolygonsec{m}_{\indextj{k}{\ell}}$ adjacent to $\junctionDiscSec{m}{k}$,
we set
\[
\widehat{\bv\nu^m_{k,\ell}}:=R_{-\pi/2}\widehat{\bv t^m_{k,\ell}}
  \sentence{with}
  \widehat{\bv t^m_{k,\ell}}
  :=\frac{\bv Q^m_{k,\ell}-\junctionDiscSec{m}{k}}
         {|\bv Q^m_{k,\ell}-\junctionDiscSec{m}{k}|},
  \sentence{for}
  \ell = 1,\ 2,\ 3.
\]
Here, $\widehat{\bv\nu^m_{k,\ell}}$ is built from the inward co-normal at the junction, that is,
from the unit tangent pointing from $\junctionDiscSec{m}{k}$ into
$\curvepolygonsec{m}_{\indextj{k}{\ell}}$, so that the endpoint signs
$\varepsilon_{k,\ell}$ of the Young--Herring law in \eqref{eq:strong} are absorbed up to a
common sign, which the modulus in \eqref{eq:strong:RH} discards; this uses that each
$\gamma_i$ is even, whence $\bv\xi_i$ is odd.
The terminal Herring residuals on the four triple junctions are
$(6.00\times10^{-3},4.54\times10^{-2},7.00\times10^{-2},6.21\times10^{-2})$.

We finally apply the proposed scheme to a larger network with
$I_P=I_R=8$, $I_S=18$, and $I_T=12$.
The initial configuration is strongly deformed, and the enclosed areas are unequal;
the largest bounded phase being $1.51$ times the smallest one.
The number of edges is $(11,16,14,10,18,12)$ on the six curves which surround the inner domain,
$14$ on each of the six curves which separate intermediate domains,
and $(30,36,33,28,38,31)$ on the six curves between intermediate domains and the unbounded phase.
We refer the reader to Figure~\ref{fig:sc7:e} for the anisotropies assigned to each curve.
We implement $200$ iterations with $\dt=1\times10^{-5}$
and observe that $\Delta E_{\max}=-1.34\times10^{-5}$, and thus the anisotropic surface energy is non-increasing.
The bounded-phase area drifts are $(5.19\times10^{-4},-1.27\times10^{-3},-1.25\times10^{-3},-3.96\times10^{-3},4.33\times10^{-4}, -3.39\times10^{-3},-2.57\times10^{-3})$.
The largest Herring residual \eqref{eq:strong:RH} over
the twelve triple junctions decreases from $7.78\times10^{-1}$ to $6.58\times10^{-2}$, the
smallest terminal value being $2.30\times10^{-3}$.

\begin{figure}[!t]
  \centering
  \includegraphics[width=\textwidth]{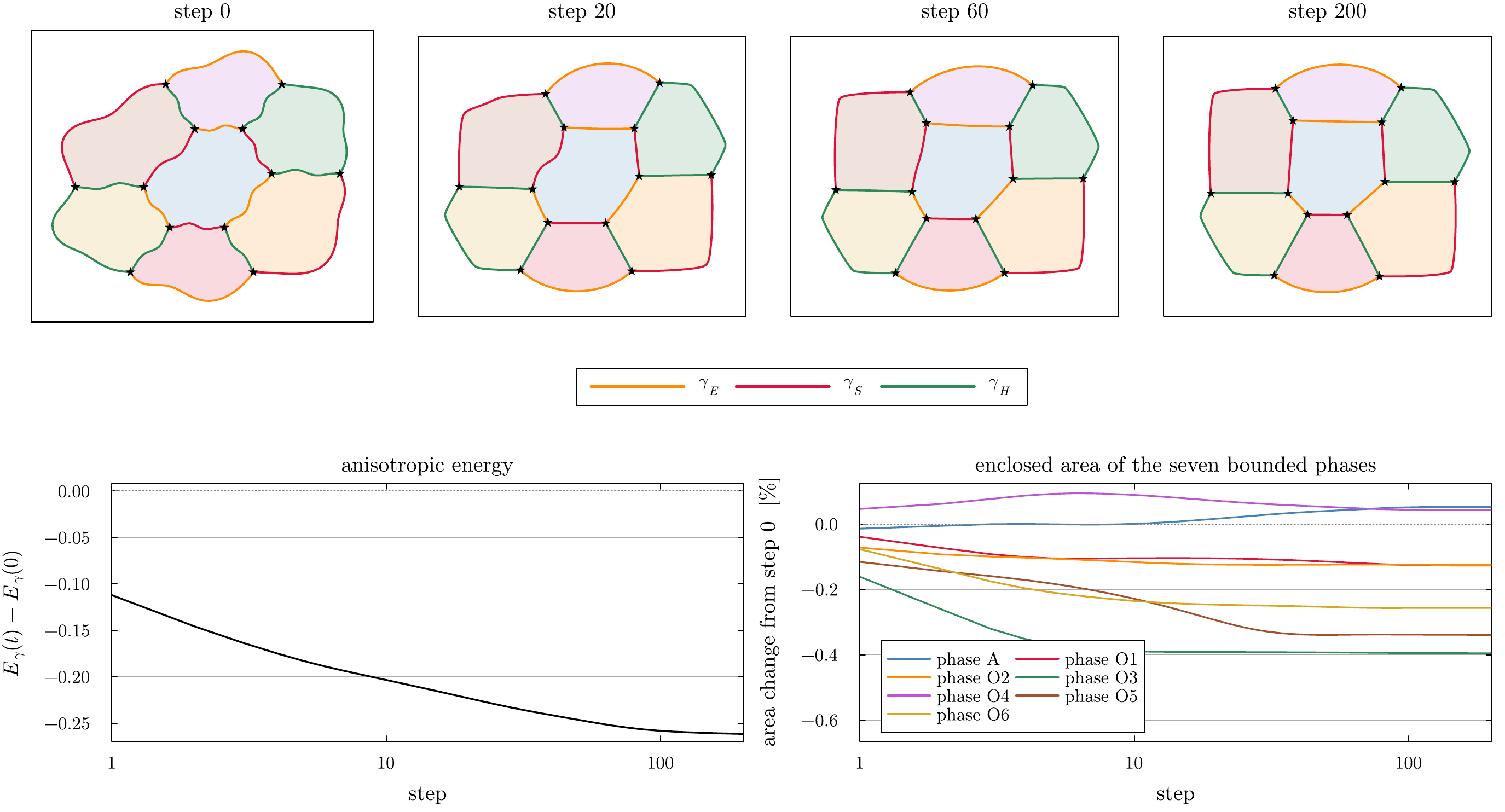}
  \caption{Seven-cell anisotropic cluster: evolution of the curve network,
  variation of anisotropic energy, and bounded-phase areas.
  }
  \label{fig:sc7:e}
\end{figure}

\subsection{Coupled refinement against a rescaled exact solution}
\label{subsec:accuracy}
Finally, we compare the proposed scheme with an exact solution of the isotropic
multi-phase Mullins--Sekerka problem, namely the network of three concentric circles
of \cite[\S8.1, (8.1a), (8.2) and Appendix~A]{EGN24}.

In this convergence test, we let $I_R=4$, $I_P=I_S=3$, and $I_T=0$.
Let $0<R_1(t)<R_2(t)<R_3(t)$ and let $\Gamma_i(t)=\partial B(\bv 0,R_i(t))$ for $1\leq i\leq3$.
The three circles bound the disk $\region{1}(t) := \{|\bv x|<R_1(t)\}$, the two annuli
$\region{2}(t) := \{R_1(t)<|\bv x|<R_2(t)\}$ and $\region{3}(t) := \{R_2(t)<|\bv x|<R_3(t)\}$,
and the unbounded exterior $\region{4}(t) := \{|\bv x|>R_3(t)\}$.
The regions $\region{1}(t)$ and $\region{4}(t)$ are occupied by the phase $0$,
while the two annuli $\region{2}(t)$ and $\region{3}(t)$ are occupied by the
phases $1$ and $2$, respectively,
and their areas are conserved in the underlying model.
We let $L_i=c_i=1$ and
$G_{i,1}$ the identity matrix in \eqref{eq:gamma}, so that
all three interfaces carry the isotropic
energy density $\gamma_i(\normal)=|\normal|$.

Setting
\[
  D_1:=R_2(0)^2-R_1(0)^2,
  \qquad
  D_2:=R_3(0)^2-R_2(0)^2,
\]
the radii solve the differential algebraic system:
\begin{equation}\label{eq:3circle-dae}
  \begin{aligned}
  R_1(t)&=\sqrt{R_2(t)^2-D_1},\\
  R_3(t)&=\sqrt{R_2(t)^2+D_2},\\
  \dot R_2(t)&=-F(R_2(t)),
  \end{aligned}
\end{equation}
where
\[
  F(u):=\left(\frac{1}{\sqrt{u^2-D_1}}+\frac1u+\frac{1}{\sqrt{u^2+D_2}}\right)
  \Big/\left(u\log\frac{u^2+D_2}{u^2-D_1}\right),
  \qquad u\in(\sqrt{D_1},\infty).
\]
The reference solution is not available in closed form.  Rather than
integrating the ordinary differential equation in
\eqref{eq:3circle-dae}, we therefore evaluate the reference radius $R_2(t)$ by a root
finding algorithm applied to the implicit relation
\[
  0=t+\int_{R_2(0)}^{R_2(t)}\frac{\mathrm du}{F(u)},
\]
which follows from \eqref{eq:3circle-dae} by separation of variables.
It has been shown in \cite[Eq.(8.1a)]{EGN24} that
the solution to \eqref{eq:3circle-dae} rigorously solves \eqref{eq:strong}, and we admit this result.

The benchmark of \cite{EGN24} uses the initial radii
$(R_1(0),R_2(0),R_3(0))=(2,2.5,3)$ and the horizon time is set to $T_0=0.5$.
We note that the target problem is invariant under the similarity law $$\widehat{\bv X_i}(t)=s\bv X_i\left(\frac{t}{s^3}\right),$$ and hence 
lengths and times may be rescaled simultaneously; for this, we use the factor $s=0.2$, and then we have
\[
  (R_1(0),R_2(0),R_3(0))=(0.4,0.5,0.6),
  \qquad T=s^3T_0=0.004.
\]
For each run, we put $N$ vertices ($N\in\{16,32,64\}$) on each circle. Let
\[
  \dt_N=\frac{T}{N^2},\qquad M_N=N^2,\qquad
  t_m=m\dt_N,\qquad M_N\dt_N=T.
\]
For all three runs, the monitored regional boundaries satisfy
\[
  r_*:=\max_{\substack{N\in\{16,32,64\}\\0\leq m\leq M_N\\1\leq\ell\leq4}}
  \max_{\bv x\in\curvepolygonsec{m}_{\region{\ell}}}|\bv x|\leq 0.6,
\]
and thus, it follows that for all $N\in\{16,32,64\}$,
\[
  \operatorname{diam}\bigl(\curvepolygonsec{m}_{\region{\ell}}\bigr)\leq 2r_* \leq 1.2<2
  \qquad 0\leq \forall m\leq M_N\sentence{and} 1\leq \forall\ell\leq 4.
\]
Hence, we deduce from Lemma~\ref{lem:diam-cap} that the hypothesis of Proposition~\ref{prop:capacity} is satisfied at every time level for all cases.
We now measure the error of radii and the error of convergence (EOC) defined by
\begin{equation}\label{eq:egamma}
  e_\Gamma(N)
  :=\max_{\substack{1\leq m\leq M_N\\1\leq i\leq3\\1\leq j\leq N}}
    \bigl||\bv X_{i,j}^{m}|-R_i(t_m)\bigr|
  \sentence{and}
  \operatorname{EOC}(N) := \log_2\frac{e_\Gamma(N/2)}{e_\Gamma(N)},
\end{equation}
where $R_i(t)\,(i = 1,\,2,\,3)$ solve the ODE \eqref{eq:3circle-dae}.
Consequently, the coupled refinement $\dt_N=O(h^2)$ gives an observed rate near two,
with energy decay at every step.
\begin{table}[!ht]
\centering
\begin{tabular}{rrlll}
\toprule
$N$ & steps & $\dt_N$ & $e_\Gamma(N)$ & $\operatorname{EOC}(N)$ \\
\midrule
$16$ & $256$  & $1.5625\times10^{-5}$ & $4.011\times10^{-3}$ & --- \\
$32$ & $1024$ & $3.9063\times10^{-6}$ & $9.543\times10^{-4}$ & $2.072$ \\
$64$ & $4096$ & $9.7656\times10^{-7}$ & $2.339\times10^{-4}$ & $2.028$ \\
\bottomrule
\end{tabular}
\caption{Convergence test:
maximum radial error and observed rate.}
\label{tab:three-circle-refinement}
\end{table}

\section{Conclusion}\label{sec:conclusion}
In this paper, we have developed a structure-preserving symmetric Galerkin boundary
element method for the anisotropic multi-phase Mullins--Sekerka flow in $\R^2$.
The bulk harmonic problems are eliminated through the interface Dirichlet-to-Neumann
operator assembled region by region, and its SGBEM realization $\dtn{\curvepolygon{}}$ inherits the three structural properties of the continuous
operator: symmetry, positive semi-definiteness, and the exact constant kernel (Proposition~\ref{prop:discrete-dtn-structure}).

For anisotropies of the square-root class \eqref{eq:gamma}, the fully discrete scheme
satisfies the energy-dissipation inequality \eqref{eq:main} for arbitrary time-step size, provided that the curve network is non-degenerate (Theorem~\ref{thm:main});
the first variation of the discrete signed area of every bounded phase vanishes at the velocity level,
so that the area defect of one step equals exactly the signed area of the displacement polygon (Propositions~\ref{prop:cons} and~\ref{prop:areadefect});
and the gauge-bordered one-step system is uniquely solvable on solvability-admissible curve networks (Theorem~\ref{prop:solv}).

The numerical experiments of Section~\ref{sec:numerics} show that
the discrete energy decays monotonically over a sweep of three orders of
magnitude in the time step, the residual of the area-defect identity
\eqref{eq:areadefect} stays at the round-off level throughout the sweep,
and the mixed-anisotropy junction networks evolve with non-increasing energy and
small phase-area drifts;
in the seven-cell test, the largest Herring residual decreases from $7.78\times10^{-1}$ to $6.58\times10^{-2}$.
Finally, a coupled space--time refinement against the three-concentric-circle exact
solution of \cite{EGN24} exhibits an observed convergence rate close to two.

The present study is restricted in several respects.
In particular, a rigorous convergence analysis of the fully discrete scheme has not been established,
and the convergence experiment is therefore limited to an isotropic configuration for which an exact solution is available.
Moreover, the analysis is carried out for anisotropies of the square-root class \eqref{eq:gamma},
and the simulations are stopped before any topological change of the curve network occurs.
The capacity condition in Proposition~\ref{prop:discrete-dtn-structure} imposes a restriction
on the absolute scale of the bounded regions, although this can always be accommodated by a similarity rescaling.

\section*{Acknowledgments}
This research is supported by the ANR (Agence Nationale de la Recherche), grant ANR-23-PEIA-0004 (PEPR IA, PDE-AI).

\bibliographystyle{plainnat}
\bibliography{cite}
\end{document}